\documentclass[12pt]{amsart}
\usepackage{amsthm,amssymb,amsmath,amscd,epic,eepic}
\usepackage[all]{xy}
\usepackage{graphicx,subcaption}
\usepackage{amsmath} 
\usepackage{epic,eepic} 
\usepackage{yfonts}
\usepackage{paralist,enumerate}
\usepackage{hyperref}
\hypersetup{colorlinks}

\usepackage{tikz}
\usetikzlibrary {positioning}
\usepackage{graphicx}
\usepackage{subcaption}
\usetikzlibrary{arrows}
\tikzset{    vertex/.style={circle,draw,minimum size=1.5em},    edge/.style={->,> = latex'}}
\usepackage{amsthm,amsfonts,amssymb,mathrsfs,amscd,amstext,amsbsy}

\newtheorem{theorem}{Theorem}[section]

\newtheorem{lemma}[theorem]{Lemma}

\newtheorem{theo}[theorem]{Theorem}
\newtheorem{lem}[theorem]{Lemma}
\newtheorem{pro}[theorem]{Proposition}

\newtheorem{exa}[theorem]{Example}
\newtheorem{que}[theorem]{Question}

\newtheorem{definition}[theorem]{Definition}
\newtheorem*{Definition*}{Definition}
\numberwithin{equation}{section}

\def\qed{\hfill \ifhmode\unskip\nobreak\fi\quad\ifmmode\Box\else$\Box$\fi\\ }

\begin{document}

\title[Almost complex torus 6-manifolds with Euler number 6]{Six dimensional almost complex torus manifolds with Euler number six}
\author{Donghoon Jang and Jiyun Park}
\thanks{Donghoon Jang was supported by the National Research Foundation of Korea(NRF) grant funded by the Korea government(MSIT) (2021R1C1C1004158).}
\address{Department of Mathematics, Pusan National University, Pusan, Korea}
\email{donghoonjang@pusan.ac.kr, pjy1147@pusan.ac.kr}

\begin{abstract}
An almost complex torus manifold is a $2n$-dimensional compact connected almost complex manifold equipped with an effective action of a real $n$-dimensional torus $T^n \simeq (S^1)^n$ that has fixed points. For an almost complex torus manifold, there is a labeled directed graph which contains information on weights at the fixed points and isotropy spheres.

Let $M$ be a 6-dimensional almost complex torus manifold with Euler number 6. We show that two types of graphs occur for $M$, and for each type of graph we construct such a manifold $M$, proving the existence. Using the graphs, we determine the Chern numbers and the Hirzebruch $\chi_y$-genus of $M$.
\end{abstract}

\maketitle

\tableofcontents

\section{Introduction}

For an almost complex torus manifold, there is a (directed labeled) graph that contains information on weights at the fixed points and isotropy spheres; a vertex of the graph corresponds to a fixed point, an edge corresponds to an isotropy sphere, and the label of an edge gives a weight at a fixed point corresponding to its vertex; see Definitions \ref{d2} and \ref{d1} and Proposition \ref{p2}.  

Suppose that two almost complex torus manifolds have the same associated graph. Then they have the same weights at the fixed points, and thus the same Chern numbers (and hence have the same Hirzebruch $\chi_y$-genus, signature, Todd genus, and Euler number.) If in addition their actions are both equivariantly formal, they have the same equivariant cohomology, because the graph is a GKM(Goresky-Kottwitz-MacPherson) graph if we forget the direction of each edge \cite[Corollary 2.4]{J4}; moreover they have the same Chern class.

In \cite{J6}, the first author showed that a $2n$-dimensional almost complex torus manifold with Euler number $n+1$, which is minimal, has the same associated graph as one for some linear action on the complex projective space $\mathbb{CP}^n$. This implies that the aforementioned invariants agree for the two manifolds. Li studied an analogous problem for Hamiltonian $S^1$-actions on symplectic manifolds with (almost) minimal fixed point sets; such a manifold shares some invariants with the complex projective spaces or Grassmannians \cite{L1, L2}.

In this paper, we study 6-dimensional almost complex torus manifolds with Euler number 6 (6 fixed points), which is the next minimal Euler number in dimension 6; since any torus action on a $(4m+2)$-dimensional compact almost complex manifold cannot have an odd number of fixed points, 6 is the next possible number of fixed points.

For a 6-dimensional almost complex torus manifold with Euler number 6, we show that two types of graphs occur as a graph associated to the manifold.

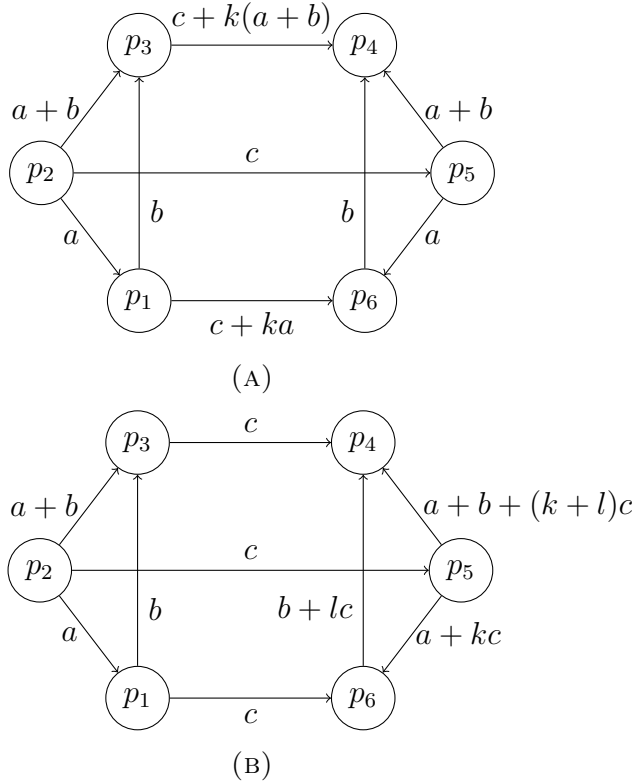
\begin{figure}
\begin{center}
\begin{subfigure}[b][5cm][s]{.53\textwidth}
\centering
\begin{tikzpicture}[state/.style ={circle, draw}]
\node[vertex] (a) at (-1.5, 0) {$p_{1}$};
\node[vertex] (b) at (-2.8, 1.7) {$p_{2}$};
\node[vertex] (c) at (-1.5, 3.4) {$p_{3}$};
\node[vertex] (d) at (1.5, 3.4) {$p_{4}$};
\node[vertex] (e) at (2.8, 1.7) {$p_{5}$};
\node[vertex] (f) at (1.5, 0) {$p_{6}$};
\path (b) [->] edge node [left] {$a$} (a);
\path (b) [->] edge node [left] {$a+b$} (c);
\path (a) [->] edge node [pos=.3, right] {$b$} (c);
\path (e) [->] edge node [right] {$a$} (f);
\path (e) [->] edge node [right] {$a+b$} (d);
\path (f) [->] edge node [pos=.3, left]{$b$} (d);
\path (a) [->] edge node [below] {$c+ka$} (f);
\path (b) [->] edge node [above] {$c$} (e);
\path (c) [->] edge node [above] {$c+k(a+b)$} (d);
\end{tikzpicture}
\caption{}\label{f4a}
\vspace{\baselineskip}
\end{subfigure}\qquad

\vspace{\baselineskip}
\begin{subfigure}[b][5cm][s]{.53\textwidth}
\centering
\begin{tikzpicture}[state/.style ={circle, draw}]
\node[vertex] (a) at (-1.5, 0) {$p_{1}$};
\node[vertex] (b) at (-2.8, 1.7) {$p_{2}$};
\node[vertex] (c) at (-1.5, 3.4) {$p_{3}$};
\node[vertex] (d) at (1.5, 3.4) {$p_{4}$};
\node[vertex] (e) at (2.8, 1.7) {$p_{5}$};
\node[vertex] (f) at (1.5, 0) {$p_{6}$};
\path (b) [->] edge node [left] {$a$} (a);
\path (b) [->] edge node [left] {$a+b$} (c);
\path (a) [->] edge node [pos=.3, right] {$b$} (c);
\path (e) [->] edge node [right=-.1]{$a+kc$} (f);
\path (e) [->] edge node [right]{$a+b+(k+l)c$} (d);
\path (f) [->] edge node [pos=.3, left] {$b+lc$} (d);
\path (a) [->] edge node [below] {$c$} (f);
\path (b) [->] edge node [above] {$c$} (e);
\path (c) [->] edge node [above] {$c$} (d);
\end{tikzpicture}
\caption{}\label{f4b}
\vspace{\baselineskip}
\end{subfigure}\qquad
\caption{Graphs for 6-dimensional almost complex torus manifolds with Euler number 6}\label{f4}
\end{center}
\end{figure}

\begin{theorem} \label{t11}
Let $M$ be a 6-dimensional almost complex torus manifold with Euler number 6. Then one of Figure \ref{f4} occurs as a graph describing $M$, for some $a$, $b$, and $c$ in $\mathbb{Z}^3$ that span $\mathbb{Z}^3$ and for some integers $k$ (and $l$ for Figure \ref{f4b}).
\end{theorem}

Moreover, for each type of graph, we construct such a manifold $M$, proving the existence.

\begin{theorem} \label{t1.2}
Each of graphs in Figure \ref{f4} describes some 6-dimensional almost complex torus manifold with Euler number 6; Figure \ref{f4a} describes Example \ref{e1} and Figure \ref{f4b} describes Example \ref{e2}.
\end{theorem}

\begin{exa} \label{e1} (More details are in Example \ref{ex1}.)
Fix an integer $k$. Let $M_1 (k)$ be a (real) 6-dimensional compact complex manifold
\begin{center}$M_1(k)=\{([z_0:z_1:z_2:z_3],[w_0:w_1:w_2]) \in \mathbb{CP}^3 \times \mathbb{CP}^2 \, | \, z_0 w_1^k=z_1 w_0^k, z_1 w_2^k=z_2 w_1^k\}$.\end{center}
Let $T^3$ act on $M_1(k)$ by
\begin{center} $g \cdot ([z_0:z_1:z_2:z_3],[w_0:w_1:w_2])$
$= ([z_0:g^{ka} z_1:g^{k(a+b)} z_2:g^{-c} z_3],[w_0: g^a w_1:g^{a+b} w_2])$
\end{center}
for all $g\in T^3 \subset \mathbb{C}^3$, where $a, b$, and $c$ form a basis of $\mathbb{Z}^3$. Figure \ref{f4a} is a graph describing $M_1(k)$ with this action.
\end{exa}

\begin{exa} \label{e2} (More details are in Example \ref{ex2}.)
Fix integers $k$ and $l$. Let $M_2 (k,l)$ be a (real) 6-dimensional compact complex manifold
\begin{center}
$M_2(k,l)=\{([z_0:z_1:z_2:z_3:z_4],[w_0:w_1]) \in \mathbb{CP}^4 \times \mathbb{CP}^1 \, | \, z_0 w_1^k=z_3 w_0^k, z_2 w_1^{k+l}=z_4 w_0^{k+l}\}$.
\end{center} 
Let $T^3$ act on $M_2 (k,l)$ by
\begin{center}$ g \cdot ([z_0:z_1:z_2:z_3:z_4],[w_0:w_1])$ 
$= ([z_0: g^{-a} z_1: g^b z_2: g^{kc} z_3: g^{b+(k+l)c} z_4],[w_0:g^c w_1])$\end{center}
for all $g\in T^3 \subset \mathbb{C}^3$, where $a$, $b$, and $c$ form a basis of $\mathbb{Z}^3$. Figure \ref{f4b} is a graph describing $M_2(k,l)$ with this action.
\end{exa}

As a consequence of Theorem \ref{t11}, for a 6-dimensional almost complex torus manifold with Euler number 6, we determine its Chern numbers and the Hirzebruch $\chi_y$-genus.

\begin{theorem} \label{t81}
Let $M$ be a 6-dimensional almost complex torus manifold with Euler number 6. Then the Hirzebruch $\chi_y$-genus of $M$ is $\chi_y (M) = 1-2y + 2y^2-y^3$.
\begin{enumerate}[(1)]
\item Suppose that Figure \ref{f4a} describes $M$. Then the Chern numbers of $M$ are $\int_M c_1 c_2 = 24$, $\int_M c_3 = 6$, and $\int_M c_1^3=2(k^2+27)$ for some integer $k$.
\item Suppose that Figure \ref{f4b} describes $M$. Then the Chern numbers of $M$ are $\int_M c_1 c_2 = 24$, $\int_M c_3 = 6$, and $\int_M c_1^3=54$.
\end{enumerate}
\end{theorem}

We note that for the manifold $M_1(k)$, if $|k|$ is large, then the Chern number $\int_M c_1^3=2(k^2+27)$ is large.

For a 4-dimensional almost complex torus manifold $M$, it need not hold that $a_0(M) \leq a_1(M)$, where $\chi_y(M)=a_0(M)-a_1(M) y+a_2(M) y^2$ is the Hirzebruch $\chi$-genus of $M$; see \cite{J3, Ma2}.

Let $M$ be a 6-dimensional almost complex torus manifold with Euler number 6. In theory, it is possible that the Hirzebruch $\chi_y$-genus of $M$ is $2-y+y^2-2y^3$; see Proposition \ref{p6}. On the other hand, Theorem \ref{t81} shows that this is not possible. It is plausible that for a 6-dimensional almost complex torus manifold $M$, if we let $\chi_y(M)=\sum_{i=0}^3 a_i(M) \cdot (-y)^i$ be the Hirzebruch $\chi_y$-genus of $M$, then $a_0(M)$ is bounded above by some increasing function of $a_1(M)$.

\begin{que}
Let $M$ be a 6-dimensional almost complex torus manifold. Then does $a_0(M) \leq f(a_1(M))$ hold for some increasing function $f(x)$? Here, $\chi_y(M)=\sum_{i=0}^3 a_i(M) \cdot (-y)^i$ is the Hirzebruch $\chi_y$-genus of $M$.
\end{que}

In the language of the theory of multi-fans \cite{Ma2}, we can rephrase this question as, if we can use $2(a_0(M)+a_1(M))$ fans to cover $\mathbb{R}^3$ $a_0(M)$-times, where $a_0(M)$ is the Todd genus of $M$ and $2(a_0(M)+a_1(M))$ is the total number of fixed points.

Let $M$ be a 6-dimensional compact almost complex manifold. In Examples \ref{e1} and \ref{e2}, we constructed examples of $M$ equipped with $T^3$-actions with 6 fixed points. On the other hand, if a smaller dimensional torus acts on $M$ with 6 fixed points, then there are more examples.
In \cite{J5}, the first author showed that if the circle group $S^1$ acts on $M$ with 4 fixed points, then 6 cases occur for the weights at the fixed points, and discussed the existence of such a manifold in each case. In particular, there is such a manifold $M$ with Todd genus 0. If we blow up a fixed point of $M$, then we get another $S^1$-manifold with 6 fixed points, whose Todd genus is 0. Such a manifold is distinct from the manifolds in Theorem \ref{t11}; the latter manifolds have Todd genus 1, as shown in Theorem \ref{t81}.

We discuss the ideas of the proof of Theorem \ref{t11}. 
The first ingredient is an isotropy submanifold, which is the set of points in a given almost complex torus manifold that are fixed by an action of some subgroup of the torus.
The second is the Hirzebruch $\chi_y$-genus and the Kosniowski formula (Theorem \ref{t4}).
The third is that a subgraph describes an almost complex torus submanifold (Proposition \ref{p7}).
The fourth is a relationship between weights at fixed points in an isotropy submanifold (Lemma \ref{l5}).

The proof of Theorem \ref{t11} goes as follows. Let $M$ be a 6-dimensional almost complex torus manifold with Euler number 6. The purpose of Theorem \ref{t11} is to determine graphs describing $M$.

First, in Lemma \ref{p3} we use the Hirzebruch $\chi_y$-genera of isotropy submanifolds and the Kosniowski formula (Theorem \ref{t4}) to obtain information on a 4-dimensional almost complex torus submanifold. 

Second, in Proposition \ref{p9} we show that Figure \ref{f1a} or Figure \ref{f1b} describes $M$.
Suppose that Figure \ref{f1b} describes $M$. Using Proposition \ref{p7} and Lemma \ref{l5}, in Proposition \ref{c1} we show that Figure \ref{f1b} does not describe $M$. 
Therefore, Figure \ref{f1a} describes $M$. 

What remains is to determine the labels of Figure \ref{f1a}. 
In Lemma \ref{c2}, using Proposition \ref{p7} and Lemma \ref{p3}, we find subgraphs describing 4-dimensional almost complex torus submanifolds. 
In Lemma \ref{l6}, we determine a graph describing a 4-dimensional almost complex torus manifold with 3 or 4 fixed points. 
Then from Lemmas \ref{c2} and \ref{l6} we show in Lemma \ref{l4} that several cases occur for relationships between the labels of the subgraphs. Then our task is to determine the labels in each case that occurs, which we do so in the proof of Theorem \ref{t11} in Section \ref{s7}. 

The structure of this paper is as follows.

In Section \ref{s2}, we review background and preliminaries. In Section \ref{s3}, we construct two 6-dimensional almost complex torus manifolds, one described by Figure \ref{f4a} (Example \ref{ex1}) and another described by Figure \ref{f4b} (Example \ref{ex2}). In Section \ref{s4}, we discuss (equivariant) blow up of an almost complex torus manifold and show that an (equivariant) blow up of a linear $T^3$-action on the complex projective space $\mathbb{CP}^3$ is the manifold $M_1(k)$ with $k=1$ that we construct in Example \ref{ex1}. 

Let $M$ be a 6-dimensional almost complex torus manifold with Euler number 6. In Section \ref{s5}, we study the Hirzebruch $\chi_y$-genus of an almost complex torus manifold. 
In Section \ref{s6}, we show that Figure \ref{f1a} describes $M$ (Propositions \ref{p9} and \ref{c1}), and we find subgraphs describing 4-dimensional almost complex torus submanifolds of $M$ (Lemma \ref{c2}). With all of these, in Section \ref{s7} we prove Theorem \ref{t11}. In Section \ref{s8}, we prove Theorem \ref{t81}; we determine the Chern numbers and the Hirzebruch $\chi_y$-genus of $M$.

\section{Background and preliminaries} \label{s2}

Let $M$ be a manifold. An \textbf{almost complex structure} on $M$ is a bundle map $J:TM \to TM$ on the tangent bundle of $M$, which restricts to a linear complex structure on each tangent space. A pair $(M,J)$ of a manifold $M$ and an almost complex structure $J$ on $M$ is called an \textbf{almost complex manifold}. By definition, any almost complex manifold has even dimension. For an action of a group $G$ on an almost complex manifold $(M,J)$, we say that the group action \textbf{preserves the almost complex structure} if $dg \circ J = J \circ dg$ for all elements $g$ in the group $G$. Throughout this paper, we assume that any group action on an almost complex manifold preserves the almost complex structure.

Let $(M,J)$ be a $2n$-dimensional almost complex manifold and let a $k$-dimensional torus $T^k$ act on $M$. Let $F$ be a fixed component and $p$ a point in $F$. The normal space $N_p F$ of $F$ at $p$ decomposes into the sum of complex 1-dimensional vector spaces $L_1, \cdots , L_{m}$, where on each $L_i$ the torus $T^k$ acts by multiplication by $g^{w_{p,i}}$ for all $g \in T^k$, for some non-zero element $w_{p,i}$ of $\mathbb{Z}^k$, $1 \leq i \leq m$. Here, $2m$ is the codimension of $F$ in $M$. These elements $w_{p,1}, \cdots, w_{p,m}$ are the same for all $p \in F$, and called the \textbf{weights} of $F$.

For a torus action on a compact almost complex manifold with isolated fixed points, there exists a labeled directed multigraph that contains information on weights at the fixed points and isotropy submanifolds \cite{J4}, which is an extension of the result for circle actions \cite{GS, JT}. The vertex set of the multigraph is the fixed point set of the manifold. If an edge $e$ of the multigraph has label $w$, then a fixed point corresponding to the initial vertex of $e$ has weight $w$, a fixed point corresponding to the terminal vertex of $e$ has weight $-w$, and the two fixed points are in the same component of an isotropy submanifold.

\begin{definition}\label{d2}
A \textbf{labeled directed k-multigraph} $\Gamma$ is a set $V$ of vertices, a set  $E$ of edges, maps $i : E\rightarrow V$ and $t : E\rightarrow V$ giving the initial and terminal vertices of each edge, and a map $w : E \rightarrow \mathbb{Z}^k$ giving the label of each edge.
\end{definition}

For an element $w = (w_1, \cdots,w_k)$ in $\mathbb{Z}^k$, by $\ker w$ we mean the subgroup of $T^k$ whose elements fix $w$. That is,
\begin{center}$\ker w = \{g=(g_1, \cdots, g_k) \in T^k \subset \mathbb{C}^k | g^w := g_1^{w_1}\cdots g_k^{w_k} =1\}$.\end{center}
If a group $G$ acts on a manifold $M$, we denote by $M^G$ its fixed point set, that is, 
\begin{center}
$M^G = \{m\in M | g\cdot m = m,\forall g\in G\}$.
\end{center}

\begin{definition}\label{d1}
Let a $k$-dimensional torus $T^k$ act on a compact almost complex manifold $M$ with isolated fixed points. We say that a (labeled directed $k$-)multigraph $\Gamma = (V, E)$ \textbf{describes} $M$ if the following hold:
\begin{enumerate}[(i)]
\item The vertex set $V$ is equal to the fixed point set $M^{T^k}$.
\item The multiset of the weights at $p$ is $\{w(e) | i(e) = p\} \cup \{-w(e) | t(e) = p\}$
for all $p \in M^{T^k}$.
\item For each edge $e$, the two endpoints $i(e)$ and $t(e)$ are in the same component of the isotropy submanifold $M^{\ker w(e)}$.
\end{enumerate}
\end{definition}

\begin{pro}\cite{J4} \label{p8}
Let a $k$-dimensional torus $T^k$ act on a compact almost complex manifold $M$ with isolated fixed points. Then there exists a (labeled directed $k$-)multigraph $\Gamma$ describing $M$ that has no self-loops.
\end{pro}

Moreover, a sub-multigraph of the multigraph describes an isotropy submanifold of $M$, where the submanifold is given an action of the torus $T^k$.

\begin{pro}\label{p7}
Let a $k$-dimensional torus $T^k$ act on a compact almost complex manifold $M$ with isolated fixed points. Let $\Gamma$ be a multigraph describing $M$ that has no self-loops. Let $S$ be a closed subgroup of $T^k$ and let $F$ be a component of the isotropy submanifold $M^S$. Let $\Gamma_F$ be a sub-multigraph of $\Gamma$ obtained as follows:
\begin{enumerate}
\item The vertex set of $\Gamma_F$ is $F \cap M^{T^k}$.
\item For $p \in \Gamma_F$, edges $e_i'$ of $p$ in $\Gamma_F$ are the edges $e_i$ of $p$ in $\Gamma$ whose labels $w_i$ satisfy $g^{w_i}=1$ for all $g \in S$.
\end{enumerate}
Then $\Gamma_F$ describes $F$, on which the $T^k$-action on $M$ restricts to act. Moreover, $\dim F=2m$ if and only if $\Gamma_F$ is $m$-valent; each vertex of $\Gamma_F$ has $m$ edges.
\end{pro}

\begin{proof}
The torus $T^k$ acts on $F$ as the restriction of the $T^k$-action on $M$ to $F$. Thus, $F^{T^k}=F \cap M^{T^k}$ and hence $\Gamma_F$ satisfies (i) of Definition \ref{d1} for the $T^k$-action on $F$.

Let $e$ be an edge of $\Gamma_F$ from $p$ to $q$ with label $w$. Because $\Gamma_F$ is a sub-multigraph of $\Gamma$, $p$ has weight $w$ and $q$ has weight $-w$. Because $g^w=1$ for all $g \in S$ by (2), the weight $w$ is a weight of the tangent space $T_pF$ to $F$ at $p$. Conversely, for $p \in F^{T^k}$, if a weight $w'$ in the tangent space $T_pM$ of $M$ satisfies $g^{w'}=1$ for all $g \in S$, then it is a weight of the tangent space $T_pF$ of $F$, and by (2) such a weight $w'$ is the label of some edge of $\Gamma_F$ at $p$. Thus, $\Gamma_F$ satisfies (ii) of Definition \ref{d1} for the $T^k$-action on $F$.

Let $e$ be an edge of $\Gamma_F$ from $p$ to $q$ with label $w$. Because $e$ is an edge of $\Gamma$, by (iii) of Definition \ref{d1}, $p$ and $q$ are in the same component $G$ of $M^{\ker w}$. By (2), $g^w=1$ for all $g \in S$. Thus, $S \subset \ker w$ and hence $G \subset M^{\ker w} \subset M^S$. Because $F$ is a connected component of $M^S$ containing $p$ and $q$ and $G$ is a connected component of $M^{\ker w} \subset M^S$ containing $p$ and $q$, this implies that $G \subset F$. 
For any $x \in G$, every weight $w'$ in $T_xG$ satisfies $g^{w'}=1$ for $g \in \ker w$. Therefore, it follows that $G \subset F^{\ker w}$. Hence, $\Gamma_F$ satisfies (iii) of Definition \ref{d1} for the $T^k$-action on $F$.

Since $\Gamma_F$ describes $F$, $\dim F=2m$ if and only if $\Gamma_F$ is $m$-valent, by (ii) of Definition \ref{d1}. \end{proof}

The following example illustrates Proposition \ref{p7}.

\begin{exa}
Let a 2-dimensional torus $T^2$ act on $\mathbb{CP}^3$ by
\begin{center}
$\displaystyle (g_1, g_2) \cdot [z_0:z_1:z_2:z_3]=[z_0:g_1z_1:g_1^2z_2:g_2z_3]$ 
\end{center}
for all $(g_1,g_2)\in T^2$ and for all $[z_0:z_1:z_2:z_3] \in \mathbb{CP}^3$. The action has 4 fixed points, $p_1=[1:0:0:0]$, $p_2=[0:1:0:0]$, $p_3=[0:0:1:0]$, and $p_4=[0:0:0:1]$ that have weights $\{(1,0), (2,0), (0,1)\}$, $\{(-1,0), (1,0), (-1,1)\}$, $\{(-2,0), (-1,0), (-2,1)\}$, and $\{(-2,0),$ $(1,-1),$ $(2,-1)\}$, respectively. Figure \ref{f14a} describes this action on $\mathbb{CP}^3$.

Consider the subgroup $\{1\}\times S^1$ of $T^2$. Its fixed point set is $\{[z_0:z_1:z_2:0]\} \cup [0:0:0:1]$. The 2-dimensional torus $T^2$ acts on the former fixed component $F_0:=\{[z_0:z_1:z_2:0]\}$ by
\begin{center} $\displaystyle (g_1,g_2)\cdot[z_0:z_1:z_2:0]=[z_0:g_1z_1:g_1^2z_2:0]$ \end{center}
for all $(g_1,g_2)\in T^2$, having fixed points $p_1=[1:0:0:0]$, $p_2=[0:1:0:0]$, and $p_3=[0:0:1:0]$ whose weights are $\{(1,0),(2,0)\}$, $\{(-1,0),(1,0)\}$, and $\{(-2,0),(-1,0)\}$, respectively. Thus, Figure \ref{f14b} describes the $T^2$-action on $F_0$.

Next, consider the subgroup $\mathbb Z_2\times \{1\}$ of $T^2$, where $\mathbb Z_2=\{1,-1\}$. Its fixed point set is $\{[z_0:0:z_2:z_3]\}\cup[0:1:0:0]$. The 2-dimensional torus $T^2$ acts on the former fixed component $F_1:=\{[z_0:0:z_2:z_3]\}$ by
\begin{center} $\displaystyle (g_1,g_2) \cdot [z_0:0:z_2:z_3]=[z_0:0:g_1^2z_2:g_2z_3]$\end{center}
for all $(g_1,g_2)\in T^2$, having fixed points $p_1=[1:0:0:0]$, $p_3=[0:0:1:0]$, and $p_4=[0:0:0:1]$ whose weights are $\{(2,0),(0,1)\}$, $\{(-2,0),(-2,1)\}$, and $\{(0,-1),(2,-1)\}$, respectively. Thus, Figure \ref{f14c} describes the $T^2$-action on $F_1$.

\begin{figure}
\begin{center}
\begin{subfigure}[b][6.5cm][s]{.30\textwidth}
\centering
\begin{tikzpicture}[state/.style ={circle, draw}]
\node[vertex] (a) at (0, 0) {$p_{1}$};
\node[vertex] (b) at (1.5, 1.7) {$p_{2}$};
\node[vertex] (c) at (1.5, 3.4) {$p_{3}$};
\node[vertex] (d) at (0, 5.1) {$p_{4}$};
\path (a) [->] edge node [right] {$(1,0)$} (b);
\path (b) [->] edge node [right] {$(1,0)$} (c);
\path (a) [->] edge node [left] {$(2,0)$} (c);
\path (a) [->] edge node [left] {$(0,1)$} (d);
\path (b) [->] edge node [left] {$(-1,1)$} (d);
\path (c) [->] edge node [right] {$(-2,1)$} (d);
\end{tikzpicture}
\caption{A multigraph describing $\mathbb{CP}^3$}\label{f14a}
\vspace{\baselineskip}
\end{subfigure}
~
\begin{subfigure}[b][6.5cm][s]{.30\textwidth}
\centering
\vspace{16.5mm}
\begin{tikzpicture}[state/.style ={circle, draw}]
\node[vertex] (a) at (0, 0) {$p_{1}$};
\node[vertex] (b) at (1.5, 1.7) {$p_{2}$};
\node[vertex] (c) at (1.5, 3.4) {$p_{3}$};
\path (a) [->] edge node [right] {$(1,0)$} (b);
\path (b) [->] edge node [right] {$(1,0)$} (c);
\path (a) [->] edge node [left] {$(2,0)$} (c);
\end{tikzpicture}
\caption{A sub-multigraph describing $F_0$}\label{f14b}
\vspace{\baselineskip}
\end{subfigure}\qquad
\vspace{\baselineskip}
\begin{subfigure}[b][6.5cm][s]{.30\textwidth}
\centering
\begin{tikzpicture}[state/.style ={circle, draw}]
\node[vertex] (a) at (0, 0) {$p_{1}$};
\node[vertex] (c) at (1.5, 3.4) {$p_{3}$};
\node[vertex] (d) at (0, 5.1) {$p_{4}$};
\path (a) [->] edge node [right] {$(2,0)$} (c);
\path (a) [->] edge node [left] {$(0,1)$} (d);
\path (c) [->] edge node [right] {$(-2,1)$} (d);
\end{tikzpicture}
\caption{A sub-multigraph describing $F_1$}\label{f14c}
\vspace{\baselineskip}
\end{subfigure}
\caption{}\label{f14}
\end{center}
\end{figure}
\end{exa}

For an almost complex torus manifold, a multigraph describing it is a graph without any self-loops. 

\begin{pro}\cite{J4}\label{p2}
Let $M$ be an almost complex torus manifold. Then there is a multigraph describing $M$ that has no multiple edges and no self-loops.
\end{pro}

Moreover, each edge corresponds to an isotropy sphere.

\begin{pro}\cite{J4}
Let $M$ be a $2n$-dimensional almost complex torus manifold. Then for each fixed point $p$, there are $n$ 2-spheres on each of which the $T^n$-action on $M$ restricts to act, sharing $p$ as one fixed point in common with weight $w_{p,i}$, but no two of them share other fixed points. Here, $w_{p,1}, \cdots, w_{p,n}$ are the weights at $p$.
\end{pro}

An almost complex manifold $M$ has at least $\frac{1}{2}\dim M+1$ fixed points \cite{Da, Ma2}. In other words, the Euler number of an almost complex torus manifold $M$ is at least $\frac{1}{2}\dim M+1$, because $M$ has finitely many fixed points, the Euler number of $M$ is equal to the sum of the Euler numbers of fixed points (Theorem \ref{t3}), and the Euler number of a point is 1. Proposition \ref{p2} implies that a graph describing $M$ has at least $\frac{1}{2}\dim M+1$ vertices, proving the same result.

\begin{pro}\cite{Da, J4, Ma2}\label{p4}
Let $M$ be an almost complex torus manifold. Then there are at least $\frac{1}{2}\dim M + 1$ fixed points. Alternatively, the Euler number of $M$ is at least $\frac{1}{2}\dim M + 1$.
\end{pro}

Let a $k$-dimensional torus $T^k$ act on a manifold $M$. The \textbf{equivariant cohomology} of $M$ is
\begin{center}
$H_{T^k}^\ast (M)=H^*(M \times_{T^k} ET^k)$,
\end{center}
where $ET^k$ is a contractible space on which $T^k$ acts freely. Suppose that $M$ is compact and oriented. Then the projection map $\pi : M \times_{T^k} ET^k \to BT^k$ induces a push-forward map
\begin{center} $\pi_* : H_{T^k}^i (M; \mathbb{Z}) \to H^{i-\dim M}(BT^k;\mathbb{Z})$\end{center}
for all $i \in \mathbb{Z}$, where $BT^k$ is the classifying space of $T^k$. The projection map $\pi_*$ is also denoted by $\int_M$, and is given by integration over the fiber $M$.

\begin{theo}[The Atiyah-Bott-Berline-Vergne localization theorem]\cite{AB}\label{ab}
Let a $k$-dimensional torus $T^k$ act on a compact oriented manifold $M$. Let $\alpha \in H_{T^k}^*(M;\mathbb{Q})$. As an element of $\mathbb{Q}(t)$,
\begin{center}$\displaystyle \int_M \alpha = \sum_{F\subset M^{T^k}} \int_F \frac{\alpha|_F}{e_{T^k}(NF)}$,\end{center}
where the sum is taken over all fixed components, and $e_{T^k}(NF)$ is the equivariant Euler class of the normal bundle to $F$ in $M$.
\end{theo}

A \textbf{unitary manifold} is a manifold with a complex structure on the bundle $TM \oplus \underline{\mathbb{R}}^k$ for some $k \geq 0$, where $\underline{\mathbb{R}}^k$ denotes the $k$-dimensional trivial bundle over $M$. By definition, any almost complex manifold is unitary.

For a torus action on a unitary manifold, at each fixed component weights of the normal bundle are well-defined. Moreover, weights at fixed points in an isotropy submanifold have an intimate relation.

\begin{lemma}\cite{J4}\label{l5}
Fix an element $w$ in $\mathbb{Z}^k$. Let a k-dimensional torus $T^k$ act on a compact unitary manifold $M$. Let $p$ and $p^\prime$ be fixed points which lie in the same component of $M^{\ker w}$. Then the $T^k$-weights at $p$ and at $p^\prime$ are equal modulo $w$.
\end{lemma}

Since any almost complex manifold is unitary, Lemma \ref{l5} holds for almost complex manifolds. 

Circle actions on almost complex manifolds with few fixed points are classified.

\begin{theorem} \label{t211}
Let the circle act on a compact almost complex manifold $M$.
\begin{enumerate}
\item If there is exactly one fixed point $p$, then $M$ is the point itself, that is, $M=\{p\}$.
\item \cite{Kos2, J2} If there are exactly two fixed points, then either $M$ is the 2-sphere, or $\dim M=6$ and weights at the fixed points are $\{-a-b,a,b\}$ and $\{-a,-b,a+b\}$ for some positive integers $a$ and $b$.
\item \cite{J2} If there are exactly three fixed points, then $\dim M=4$ and weights at the fixed points are $\{a+b,a\}$, $\{-a,b\}$, and $\{-b,-a-b\}$ for some positive integers $a$ and $b$.
\end{enumerate}
\end{theorem}

The proof for the case of 3 fixed points in \cite{J2} is a careful modification of the proof for a symplectic action in \cite{J1}.

\section{Existence: examples} \label{s3}

By Theorem \ref{t11}, any 6-dimensional almost complex torus manifold with Euler number 6 is described by Figure \ref{f4a} or Figure \ref{f4b}. In this section, we construct examples of almost complex torus manifolds, one described by Figure \ref{f4a} and the other described by Figure \ref{f4b}. 

We note that Example \ref{ex1} with $k=0$ is the same as Example \ref{ex2} with $k=l=0$. We also note that the manifold constructed in \cite[Example 7.1]{J4} is a special case of Example \ref{ex2} that we construct in this paper with $k=0$. Moreover, Example \ref{ex1} with $k=1$ is also obtained as blowing up at a fixed point of a linear $T^3$-action on $\mathbb{CP}^3$; see Example \ref{ex3} for details.

The first example is an almost complex torus manifold described by Figure \ref{f4a}.

\begin{exa} \label{ex1}
Fix an integer $k$. By $M_1 (k)$ we mean a compact complex manifold
\begin{center}$M_1(k)=\{([z_0:z_1:z_2:z_3],[w_0:w_1:w_2]) \in \mathbb{CP}^3 \times \mathbb{CP}^2 \, | \, z_0 w_1^k=z_1 w_0^k, z_1 w_2^k=z_2 w_1^k\}$.\end{center}
Note that $z_0 w_1^k=z_1 w_0^k$ and $z_1 w_2^k=z_2 w_1^k$ imply that $z_0 w_2^k=z_2 w_0^k$.

Let $a$, $b$, and $c$ form a basis of $\mathbb{Z}^3$. Let $T^3$ act on $M_1(k)$ by
\begin{center}
$g \cdot ([z_0:z_1:z_2:z_3],[w_0:w_1:w_2])$

$=([z_0:g^{ka} z_1:g^{k(a+b)} z_2:g^{-c} z_3],[w_0: g^a w_1:g^{a+b} w_2])$
\end{center}
for all $g\in T^3 \subset \mathbb{C}^3$.
The action has 6 fixed points $p_1=([1:0:0:0],[1:0:0])$, $p_2=([0:1:0:0],[0:1:0])$, $p_3=([0:0:1:0],[0:0:1])$, $p_4= ([0:0:0:1],[1:0:0])$, $p_5=([0:0:0:1],[0:1:0])$, and $p_6=([0:0:0:1],[0:0:1])$ that have weights $\{-c, a, a+b\}$, $\{-c-ka, -a, b\}$, $\{-c-k(a+b), -a-b, -b\}$, $\{c, a, a+b\}$, $\{c+ka, -a, b\}$, and $\{c+k(a+b), -a-b, -b\}$, respectively.

For instance, near $p_1$, we have $z_0 \neq 0$, $w_0 \neq 0$, $\frac{z_1}{z_0}=(\frac{w_1}{w_0})^k$, and $\frac{z_2}{z_0}=(\frac{w_2}{w_0})^k$. Thus $(\frac{z_3}{z_0}, \frac{w_1}{w_0}, \frac{w_2}{w_0})$ are local coordinates at $p_1$. Then $T^3$ acts on a neighborhood of $p_1$ by
\begin{center}
$\displaystyle g \cdot \Big( \frac{z_3}{z_0}, \frac{w_1}{w_0}, \frac{w_2}{w_0} \Big) = \Big( \frac{g^{-c} z_3}{z_0}, \frac{g^aw_1}{w_0}, \frac{g^{a+b} w_2}{w_0} \Big) = \Big( g^{-c}\frac{z_3}{z_0}, g^a\frac{w_1}{w_0}, g^{a+b}\frac{w_2}{w_0} \Big)$.
\end{center}
Therefore, the weights at $p_1$ are $\{-c, a, a+b\}$.

We have that $p_1$ has weight $a$, $p_2$ has weight $-a$, and $p_1$ and $p_2$ are in the same 2-sphere $F_{1,2}:= \{([z_0 : z_1 : 0 : 0], [w_0 : w_1 : 0]) | z_0 w_1^k = z_1 w_0^k\}$ of $M^{\ker a}$. Thus, in a graph describing $M_1(k)$ with this action, there is an edge from $p_1$ to $p_2$ with label $a$.

By similar computations, one can see that Figure \ref{f4a} describes $M_1(k)$ with this action.
\end{exa}

The second example is an almost complex torus manifold described by Figure \ref{f4b}.

\begin{exa} \label{ex2}
Fix integers $k$ and $l$. By $M_2 (k,l)$ we mean a compact complex manifold
\begin{center}$M_2(k,l)=\{([z_0:z_1:z_2:z_3:z_4],[w_0:w_1]) \in \mathbb{CP}^4 \times \mathbb{CP}^1 \, | \, z_0 w_1^k=z_3 w_0^k, z_2 w_1^{k+l}=z_4 w_0^{k+l}\}$.\end{center} 
Let $a$, $b$, and $c$ form a basis of $\mathbb{Z}^3$.
Let $T^3$ act on $M_2 (k,l)$ by
\begin{center}$ g \cdot ([z_0:z_1:z_2:z_3:z_4],[w_0:w_1])$ 
$= ([z_0: g^{-a} z_1: g^b z_2: g^{kc} z_3: g^{b+(k+l)c} z_4],[w_0:g^c w_1])$\end{center}
for all $g\in T^3 \subset \mathbb{C}^3$. The action has 6 fixed points $p_1=([1:0:0:0:0],[1:0])$, $p_2=([0:1:0:0:0],[1:0])$, $p_3=([0:1:0:0:0],[0:1])$, $p_4= ([0:0:1:0:0],[1:0])$, $p_5=([0:0:0:1:0],[0:1])$, and $p_6=([0:0:0:0:1],[0:1])$ that have weights $\{-a, b, c\}$, $\{a, a+b, c\}$, $\{a+kc, a+b+(k+l)c, -c\}$, $\{-b, -a-b, c\}$, $\{-a-kc, b+lc, -c\}$, and $\{-a-b-(k+l)c, -b-lc, -c\}$, respectively.

For instance, near $p_3$, we have $z_1 \neq 0$, $w_1 \neq 0$, $\frac{z_0}{z_1}=\frac{z_3}{z_1}(\frac{w_0}{w_1})^k$, and $\frac{z_2}{z_1}=\frac{z_4}{z_1}(\frac{w_0}{w_1})^{k+l}$. Thus, $(\frac{z_3}{z_1},\frac{z_4}{z_1},\frac{w_0}{w_1})$ are local coordinates at $p_3$. Then $T^3$ acts on a neighborhood of $p_3$ by
\begin{center}$\displaystyle g\cdot \Big(\frac{z_3}{z_1},\frac{z_4}{z_1},\frac{w_0}{w_1}\Big)=\Big(\frac{g^{kc}z_3}{g^{-a}z_1},\frac{g^{b+(k+l)c}z_4}{g^{-a}z_1},\frac{w_0}{g^c w_1}\Big) = \Big(g^{a+kc}\frac{z_3}{z_1},g^{a+b+(k+l)c}\frac{z_4}{z_1},g^{-c}\frac{w_0}{w_1}\Big)$.\end{center} Therefore, the weights at $p_3$ are $\{a+kc, a+b+(k+l)c, -c\}$.

We have that $p_3$ has weight $a+b+(k+l)c$, $p_6$ has weight $-a-b-(k+l)c$, and $p_3$ and $p_6$ are in the same 2-sphere $F_{3,6}^\prime:= \{([0 : z_1 : 0 : 0 : z_4], [0 : 1])\}$ of $M^{\ker a+b+(k+l)c}$. Thus, in a graph describing $M_2(k,l)$ with this action, there is an edge from $p_3$ to $p_6$ with label $a+b+(k+l)c$.

By similar computations, one can see that Figure \ref{f4b} describes $M_2(k,l)$ with this action.
\end{exa}

\section{Blow up} \label{s4}

In this section, we discuss blow up at a fixed point of an almost complex torus manifold and a graph describing it. We shall focus on dimension 6.

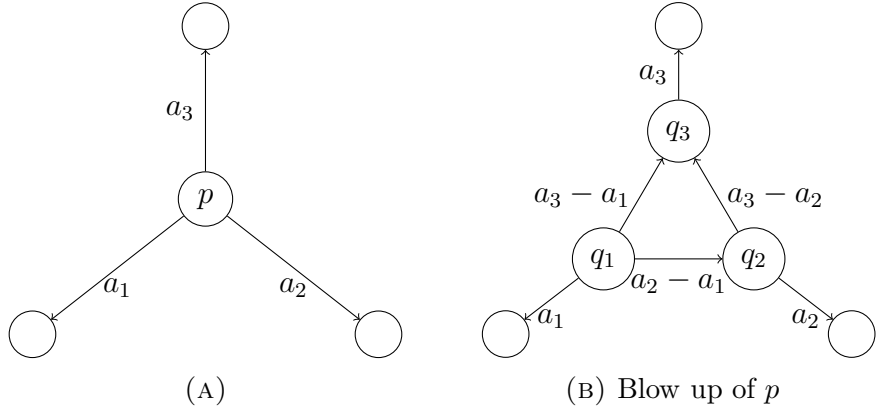
\begin{figure}
\begin{center}
\begin{subfigure}[b][5cm][s]{.45\textwidth}
\centering
\begin{tikzpicture}[state/.style ={circle, draw}]
\node[vertex] (a) at (0, 0) {$p$};
\node[vertex] (b) at (-2.3, -1.8) {};
\node[vertex] (c) at (2.3,-1.8) {};
\node[vertex] (d) at (0, 2.3) {};
\path (a) [->] edge node [below] {$a_1$} (b);
\path (a) [->] edge node [below] {$a_2$} (c);
\path (a) [->] edge node [left] {$a_3$} (d);
\end{tikzpicture}
\caption{}\label{f10a}
\vspace{\baselineskip}
\end{subfigure}\qquad
\begin{subfigure}[b][5cm][s]{.4\textwidth}
\centering
\begin{tikzpicture}[state/.style ={circle, draw}]
\node[vertex] (e) at (-1, -1) {$q_1$};
\node[vertex] (f) at (1, -1) {$q_2$};
\node[vertex] (g) at (0, .7) {$q_3$};
\node[vertex] (b) at (-2.3, -2) {};
\node[vertex] (c) at (2.3,-2) {};
\node[vertex] (d) at (0, 2.1) {};
\path (e) [->] edge node [below] {$a_1$} (b);
\path (f) [->] edge node [below] {$a_2$} (c);
\path (g) [->] edge node [left] {$a_3$} (d);
\path (e) [->] edge node [below] {$a_2-a_1$} (f);
\path (f) [->] edge node [right] {$a_3-a_2$} (g);
\path (e) [->] edge node [left] {$a_3-a_1$} (g);
\end{tikzpicture}
\caption{Blow up of $p$}\label{f10b}
\vspace{\baselineskip}
\end{subfigure}\qquad
\vspace{\baselineskip}
\caption{Blow up of $M$ and its graph}\label{f10}
\end{center}
\end{figure}

Let $M$ be a 6-dimensional almost complex torus manifold. Let $p$ be a fixed point and let $a_1$, $a_2$, and $a_3$ be the weights at $p$. Then Figure \ref{f10a} describes $M$ near $p$. Now we shall blow up $p$. We can identify a neighborhood of $p$ with a neighborhood of the origin in $\mathbb{C}^3$, on which the torus $T^3$ acts by
\begin{center}
$g \cdot (z_1,z_2,z_3)=(g^{a_1} z_1, g^{a_2} z_2, g^{a_3} z_3)$    
\end{center}
for all $g \in T^3$ and $(z_1,z_2,z_3) \in \mathbb{C}^3$. In this description, blowing up $p$ corresponds to blowing up the origin in $\mathbb{C}^3$, which is to replace the origin in $\mathbb{C}^3$ with all complex straight lines through it. Near $p$, the blown up space is described by
\begin{center}
$\{((z_1,z_2,z_3),[w_1,w_2,w_3]) \in \mathbb{C}^3 \times \mathbb{CP}^2 | z_i w_j=z_j w_i, \forall i,j\}$.
\end{center}
The $T^3$-action near $p$ naturally extends to act on the blown up space by
\begin{center}
$g \cdot ((z_1,z_2,z_3),[w_1,w_2,w_3])=((g^{a_1} z_1, g^{a_2} z_2,g^{a_3} z_3),[g^{a_1} w_1, g^{a_2} w_2, g^{a_3} w_3])$,
\end{center}
which has 3 fixed points $q_1=((0,0,0),[1:0:0])$, $q_2=((0,0,0),[0:1:0])$, and $q_3=((0,0,0),[0:0:1])$ that have weights $\{a_1,a_2-a_1,a_3-a_1\}$, $\{a_2,a_1-a_2,a_3-a_2\}$, and $\{a_3,a_1-a_3,a_2-a_3\}$, respectively. Let $\widetilde{M}$ denote the blown up manifold. Then Figure \ref{f10b} describes $\widetilde{M}$ near the blow up of $p$.

\begin{exa}\label{ex3}
Let $a_1$, $a_2$, and $a_3$ form a basis of $\mathbb{Z}^3$. Let $T^3$ act on $\mathbb{CP}^3$ by
\begin{center} $g \cdot [z_0:z_1:z_2:z_3] = [z_0:g^{a_1}z_1:g^{a_2}z_2:g^{a_3}z_3]$\end{center}
for all $g\in T^3 \subset \mathbb{C}^3$ and for all $[z_0:z_1:z_2:z_3] \in \mathbb{CP}^3$.  
Then both Figure \ref{f8a} and Figure \ref{f8b} describe this linear standard $T^3$-action on $\mathbb{CP}^3$ as they are the same graph up to position of vertices. The fixed point $p_0=[1:0:0:0]$ has weights $\{a_1,a_2,a_3\}$. Let $\widetilde{\mathbb{CP}^3}$ denote blowing up of $\mathbb{CP}^3$ at $p_0$. As discussed above, in the blown up manifold $\widetilde{\mathbb{CP}^3}$ the fixed point $p_0$ is replaced with three fixed points $p_4$, $p_5$, and $p_6$ that have weights $\{a_1,a_2-a_1,a_3-a_1\}$, $\{a_2,a_1-a_2,a_3-a_2\}$, and $\{a_3,a_1-a_3,a_2-a_3\}$, respectively. Thus, Figure \ref{f9} describes $\widetilde{\mathbb{CP}^3}$.

If we replace $a_2-a_1$, $a_3-a_2$, and $a_1$ with $a$, $b$, and $c$, respectively, then the $T^3$-action on $\widetilde{\mathbb{CP}^3}$ is the same as the action on $M_1(k)$ with $k=1$ in Example \ref{ex1}, and Figure \ref{f9} is the same as Figure \ref{f4a} with $k=1$.
\end{exa}

\begin{figure}
\begin{center}
\begin{subfigure}[b][5cm][s]{.4\textwidth}
\centering
\begin{tikzpicture}[state/.style ={circle, draw}]
\node[vertex] (a) at (0, 0) {$p_{0}$};
\node[vertex] (b) at (2, 1.5) {$p_{1}$};
\node[vertex] (c) at (2,3.5) {$p_{2}$};
\node[vertex] (d) at (0, 5) {$p_{3}$};
\path (a) [->] edge node [below] {$a_1$} (b);
\path (a) [->] edge node [left] {$a_2$} (c);
\path (a) [->] edge node [left] {$a_3$} (d);
\path (b) [->] edge node [right] {$a_2-a_1$} (c);
\path (b) [->] edge node [left] {$a_3-a_1$} (d);
\path (c) [->] edge node [pos=.2, above=.4] {$a_3-a_2$} (d);
\end{tikzpicture}
\caption{}\label{f8a}
\vspace{\baselineskip}
\end{subfigure}\qquad
\vspace{\baselineskip}
\begin{subfigure}[b][5cm][s]{.4\textwidth}
\centering
\begin{tikzpicture}[state/.style ={circle, draw}]
\node[vertex] (a) at (0, 0) {$p_{0}$};
\node[vertex] (b) at (-2.5, -1.8) {$p_{1}$};
\node[vertex] (c) at (2.5,-1.8) {$p_{2}$};
\node[vertex] (d) at (0, 2.5) {$p_{3}$};
\path (a) [->] edge node [below] {$a_1$} (b);
\path (a) [->] edge node [below] {$a_2$} (c);
\path (a) [->] edge node [left] {$a_3$} (d);
\path (b) [->] edge node [below] {$a_2-a_1$} (c);
\path (b) [->] edge node [left] {$a_3-a_1$} (d);
\path (c) [->] edge node [right] {$a_3-a_2$} (d);
\end{tikzpicture}
\caption{}\label{f8b}
\vspace{\baselineskip}
\end{subfigure}\qquad
\vspace{\baselineskip}
\caption{Graph describing a standard $T^3$-action on $\mathbb{CP}^3$}\label{f8}
\end{center}
\end{figure}

\begin{figure}
\begin{center}
\begin{subfigure}[b][6cm][s]{.55\textwidth}
\centering
\begin{tikzpicture}[state/.style ={circle, draw}]
\node[vertex] (e) at (-1, -1) {$p_{4}$};
\node[vertex] (f) at (1, -1) {$p_{5}$};
\node[vertex] (g) at (0, .7) {$p_{6}$};
\node[vertex] (b) at (-3.5, -2.5) {$p_{1}$};
\node[vertex] (c) at (3.5,-2.5) {$p_{2}$};
\node[vertex] (d) at (0, 3) {$p_{3}$};
\path (e) [->] edge node [below] {$a_1$} (b);
\path (f) [->] edge node [below] {$a_2$} (c);
\path (g) [->] edge node [left] {$a_3$} (d);
\path (b) [->] edge node [below] {$a_2-a_1$} (c);
\path (b) [->] edge node [left] {$a_3-a_1$} (d);
\path (c) [->] edge node [right] {$a_3-a_2$} (d);
\path (e) [->] edge node [below] {$a_2-a_1$} (f);
\path (f) [->] edge node [right] {$a_3-a_2$} (g);
\path (e) [->] edge node [left] {$a_3-a_1$} (g);
\end{tikzpicture}
\vspace{\baselineskip}
\end{subfigure}\qquad
\vspace{\baselineskip}
\caption{}\label{f9}
\end{center}
\end{figure}

\section{Hirzebruch genera} \label{s5}

For a compact unitary manifold, the \textbf{Hirzebruch $\chi_y$-genus}, the \textbf{Euler class}, the \textbf{Todd genus}, and the \textbf{$L$-genus} are the genera of power series $\frac{x(1+ye^{-x(1+y)})}{1-e^{-x(1+y)}}$, $1+x$, $\frac{x}{1-e^{-x}}$, and $\frac{\sqrt{x}}{\tanh \sqrt{x}}$, respectively; the latter three power series are obtained from the first power series by evaluating at $y=-1$, $y=0$, and $y=1$. Therefore, the Hirzebruch $\chi_y$-genus contains information on the Euler number, the Todd genus, and the $L$-genus.

Let the circle group $S^1$ act on a unitary manifold. For a fixed component $F$, we denote by $d(-,F)$ ($d(+,F)$) the number of negative (positive) weights in the normal bundle of $F$.

\begin{theo}[Kosniowski formula]\cite{HT, KY, Kos1}\label{t4}
Let the circle act on a compact unitary manifold $M$. Then the Hirzebruch $\chi_y$-genus $\chi_y (M)$ of $M$ satisfies
\begin{center}
$\displaystyle \chi_y (M)= \sum_{F\subset M^{S^1}} (-y)^{d(-,F)}\cdot \chi_y(F) = \sum_{F\subset M^{S^1}} (-y)^{d(+,F)}\cdot \chi_y(F)$,
\end{center}
where the sum is taken over all fixed components.
\end{theo}

Again, since any almost complex manifold is unitary, Theorem \ref{t4} holds for any compact almost complex $S^1$-manifold.
The below lemma is a simple application of the Kownioswki formula.

\begin{lem}\cite{J4}\label{l1}
Let a $k$-dimensional torus $T^k$ act on a $2n$-dimensional compact almost complex manifold $M$ with isolated fixed points. Let $\chi_y (M) = \sum_{i=0}^n a_i(M) \cdot (-y)^i$ be the Hirzebruch $\chi_y$-genus of $M$. Then the following hold:
\begin{enumerate}
\item $a_i(M) \geq 0$ for $0 \leq i \leq n$.
\item The total number of fixed points is $\sum_{i=0}^n a_i(M)$.
\item For $0\leq i \leq n$, $a_i(M) = a_{n-i}(M)$.
\end{enumerate}
\end{lem}

In \cite{J4}, the first author used the Kosniowski formula to show that for an almost complex torus manifold, the coefficients of its Hirzebruch genus are non-zero.

\begin{theo}\label{t1}\cite{J4}
Let $M$ be a $2n$-dimensional almost complex torus manifold. Then $a_i(M)>0$ for $ 0\leq i \leq n$, where $\chi_y (M) = \sum_{i=0}^n a_i(M) \cdot (-y)^i$ is the Hirzebruch $\chi_y$-genus of $M$. In particular, the Todd genus of $M$ is positive.
\end{theo}

For a compact manifold equipped with a torus action, its Euler number is equal to the Euler number of its fixed point set.

\begin{theo}\cite{Kob}\label{t3}
Let a torus $T^k$ act on a compact manifold $M$. Then the Euler number $\chi(M)$ of $M$ is equal to the sum of Euler numbers of its fixed components. That is,
\begin{center}$\displaystyle \chi(M) = \sum_{F\subset M^{T^k}}\chi(F)$.\end{center}
\end{theo}

To prove Theorem \ref{t11}, we first determine the Hirzebruch $\chi_y$-genus of a 6-dimensional almost complex torus manifold with Euler number 6.

\begin{pro}\label{p6}
Let $M$ be a 6-dimensional almost complex torus manifold with Euler number 6. Then $M$ has exactly 6 fixed points, and the Hirzebruch $\chi_y$-genus $\chi_y(M)$ of $M$ is either $\chi_y(M)=1-2y+2y^2-y^3$ or $\chi_y(M)=2-y+y^2-2y^3$.
\end{pro}

\begin{proof}
Since there are only finitely many fixed points and the Euler number of a point is 1, by Theorem \ref{t3}, $M$ has exactly 6 fixed points.

Let $\chi_y(M)=\sum_{i=0}^3 a_i(M) \cdot (-y)^i$. 
By Lemma \ref{l1} and Theorem \ref{t1}, it follows that $a_0(M)+\cdots + a_3(M)=6$, $a_0(M)=a_3(M)>0$, and $a_1(M)=a_2(M)>0$. These imply that either $a_0(M)=a_3(M)=1$ and $a_1(M)=a_2(M)=2$, or $a_0(M)=a_3(M)=2$ and $a_1(M)=a_2(M)=1$.
\end{proof}

Later, in Section \ref{s8} we shall show that in fact $\chi_y(M)=1-2y+2y^2-y^3$. With the above, one important step for proving Theorem \ref{t11} is that for a 6-dimensional almost complex torus manifold with Euler number 6, a 4-dimensional isotropy submanifold contains at most 4 fixed points.

\begin{lem}\label{p3}
Let $M$ be a 6-dimensional almost complex torus manifold with Euler number 6. Let $p$ be a fixed point and let $a, b, c$ be the weights at $p$. Let $S=\ker a \cap \ker b$ and let $F_0$ be a component of $M^S$ containing $p$. Then $F_0$ is a 4-dimensional almost complex torus submanifold and has at most 4 fixed points.
\end{lem}

\begin{proof}
By Proposition \ref{p6}, $M$ has exactly 6 fixed points. 
Since the weights at $p$ form a basis of $\mathbb{Z}^3$, $S$ is an 1-dimensional subtorus of $T^3$, that is, $S$ is a subcircle of $T^3$. 
Since $S$ only fixes $a$ and $b$ among the weights at $p$, it follows that $F_0$ is 4-dimensional. The 2-dimensional torus $T^3/S$ acts effectively on $F_0$, having $p$ as a fixed point.
Thus, $F_0$ is an almost complex torus submanifold\footnote{By Theorem \ref{t211}, $F_0$ has at least 3 fixed points.}.

Suppose on the contrary that $F_0$ has at least 5 fixed points. The $T^3$-action on $M$ restricts to act on $F_0$. Since the $T^3$-action on $F_0$ has at least 5 fixed points, applying Lemma \ref{l1} to the $T^3$-action on $F_0$, we have that $a_0(F_0)$, $a_1(F_0)$, $a_2(F_0)$ $\geq 0$, $a_0(F_0)+a_1(F_0)+a_2(F_0) \geq 5$, and $a_0(F_0)=a_2(F_0)$, where $\chi_y(F_0) = a_0(F_0)-a_1(F_0) \cdot y+a_2(F_0) \cdot y^2$. In particular, one of the following holds:
\begin{enumerate}
\item $a_0(F_0)=a_2(F_0)$ is bigger than or equal to 2.
\item $a_1(F_0)$ is bigger than or equal to 3.
\end{enumerate}

Let $F$ be any fixed component of the set $M^S$. The $T^3$-action on $F$ has isolated fixed points because its fixed point set is $F \cap M^{T^3}$. Applying Lemma \ref{l1} to the $T^3$-action on $F$, we have that $a_i(F) \geq 0$ for $0 \leq i \leq \frac{1}{2}\dim F$, 
where $\chi_y(F)=\sum_{i=0}^{\frac{1}{2}\dim F} a_i(F) \cdot (-y)^i$. 

Applying Theorem \ref{t4} to the $S$-action on $M$, we have
\begin{equation} \label{eq1}
\displaystyle \chi_y (M)= \sum_{F\subset M^{S}} (-y)^{d(-,F)}\cdot \chi_y(F),
\end{equation}
where for each fixed component $F\subset M^S$, $d(-,F)$ is the number of negative $S$-weights in the normal bundle of $F$.

Because $F_0$ is a 4-dimensional submanifold in a 6-dimensional manifold $M$, its normal bundle $NF_0$ has exactly one weight. Thus, either (a) $d(-,F_0)=0$ or (b) $d(-,F_0)=1$.

Suppose that Case (1) holds, that is, $a_0(F_0)=a_2(F_0) \geq 2$. Assume that (a) $d(-,F_0)=0$. By Proposition \ref{p6}, the Hirzebruch $\chi_y$-genus of $M$ is either (i) $\chi_y(M)=1-2y+2y^2-y^3$ or (ii) $\chi_y(M)=2-y+y^2-2y^3$. The case (i) cannot hold because $a_0(F_0) \geq 2$ and $d(-,F_0)=0$ imply $a_0(M) \geq 2$ by Equation \eqref{eq1}; other fixed components $F$ contribute the coefficients of the constant term of $\chi_y(M)$ by non-negative integers. Similarly, the case (ii) cannot hold because $a_2(F_0) \geq 2$ and $d(-,F_0)=0$ imply $a_2(M) \geq 2$ by Equation \eqref{eq1}.

Next, assume that (b) $d(-,F_0)=1$. We cannot have (i) $\chi_y(M)=1-2y+2y^2-y^3$ because $a_2(F_0) \geq 2$ and $d(-,F_0)=1$ imply that $a_3(M) \geq 2$, and we cannot have (ii) $\chi_y(M)=2-y+y^2-2y^3$ because $a_0(F_0) \geq 2$ and $d(-,F_0)=1$ imply that $a_1(M) \geq 2$.

Finally, Case (2) cannot hold; by Equation \eqref{eq1}, $a_1(F_0) \geq 3$ implies that $a_1(M)$ or $a_2(M)$ must be at least 3, but $\chi_y(M)=1-2y+2y^2-y^3$ or $\chi_y(M)=2-y+y^2-2y^3$.
\end{proof}

\section{Graphs - dimension 6 and Euler number 6} \label{s6}

We determine possible graphs describing a 6-dimensional almost complex torus manifold with Euler number 6; two possible types of graphs occur.

\begin{figure}
\begin{center}
\begin{subfigure}[b][4.2cm][s]{.4\textwidth}
\begin{tikzpicture}[state/.style ={circle, draw}]
\node[vertex] (a) at (0, 0) {$q_{2}$};
\node[vertex] (b) at (-1.2, 1.7) {$q_{1}$};
\node[vertex] (c) at (0, 3.4) {$q_{3}$};
\node[vertex] (d) at (2, 3.4) {$q_{6}$};
\node[vertex] (e) at (3.2, 1.7) {$q_{4}$};
\node[vertex] (f) at (2, 0) {$q_{5}$};
\path (b) [->] edge node [left] {$w_{1,2}$} (a);
\path (b) [->] edge node [left] {$w_{1,3}$} (c);
\path (a) [->] edge node [pos=.2] [right=-.1] {$w_{2,3}$} (c);
\path (e) [->] edge node [right] {$w_{4,5}$} (f);
\path (e) [->] edge node [right] {$w_{4,6}$} (d);
\path (f) [->] edge node [pos=.2] [left=-.1] {$w_{5,6}$} (d);
\path (a) [->] edge node [below] {$w_{2,5}$} (f);
\path (b) [->] edge node [above] {$w_{1,4}$} (e);
\path (c) [->] edge node [above] {$w_{3,6}$} (d);
\end{tikzpicture}
\caption{}\label{f1a}
\vspace{\baselineskip}
\end{subfigure}\qquad
\begin{subfigure}[b][4.3cm][s]{.4\textwidth}
\begin{tikzpicture}[state/.style ={circle, draw}]
\node[vertex] (a) at (0, 0) {$q_{2}$};
\node[vertex] (b) at (-1.2, 1.7) {$q_{1}$};
\node[vertex] (c) at (0, 3.4) {$q_{3}$};
\node[vertex] (d) at (2, 3.4) {$q_{6}$};
\node[vertex] (e) at (3.2, 1.7) {$q_{4}$};
\node[vertex] (f) at (2, 0) {$q_{5}$};
\path (b) [->] edge node [left] {$w_{1,2}$} (a);
\path (b) [->] edge node [left] {$w_{1,3}$} (c);
\path (a) [->] edge node [pos=.2] [left=-.1] {$w_{2,6}$} (d);
\path (e) [->] edge node [right] {$w_{4,5}$} (f);
\path (e) [->] edge node [right] {$w_{4,6}$} (d);
\path (c) [->] edge node [pos=.2] [right=-.1] {$w_{3,5}$} (f);
\path (a) [->] edge node [below] {$w_{2,5}$} (f);
\path (b) [->] edge node [pos=.3] [above=-.1] {$w_{1,4}$} (e);
\path (c) [->] edge node [above] {$w_{3,6}$} (d);
\end{tikzpicture}
\caption{}\label{f1b}
\vspace{\baselineskip}
\end{subfigure}\qquad
\vspace{\baselineskip}
\caption{Possible graphs}\label{f1}
\end{center}
\end{figure}
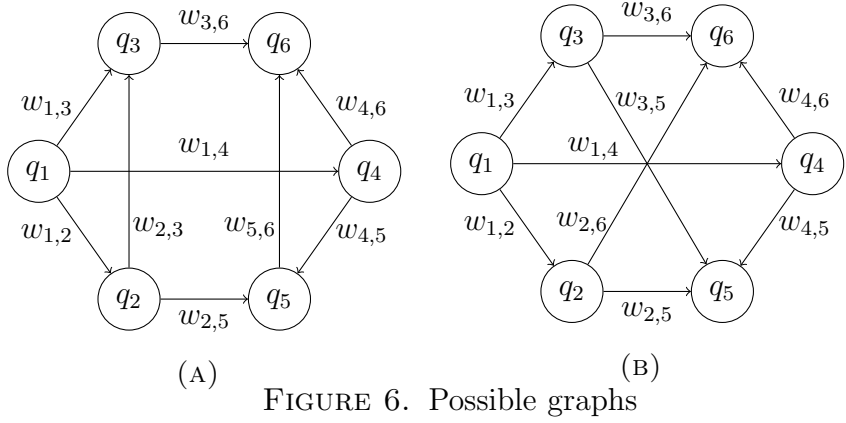

\begin{pro}\label{p9}
Let $M$ be a 6-dimensional almost complex torus manifold with Euler number 6. Then one of graphs in Figure \ref{f1} describes $M$.
\end{pro}

\begin{figure}
\begin{center}
\begin{subfigure}[b][4cm][s]{.4\textwidth}
\begin{tikzpicture}[state/.style ={circle, draw}]
\node[vertex] (a) at (0, 0) {$q_{2}$};
\node[vertex] (b) at (-1.2, 1.7) {$q_{1}$};
\node[vertex] (c) at (0, 3.4) {$q_{3}$};
\node[vertex] (d) at (2, 3.4) {$q_{6}$};
\node[vertex] (e) at (3.2, 1.7) {$q_{4}$};
\node[vertex] (f) at (2, 0) {$q_{5}$};
\path (b) [->] edge node {$$} (a);
\path (b) [->] edge node {$$} (c);
\path (b) [->] edge node {$$} (e);
\end{tikzpicture}
\vspace{\baselineskip}
\end{subfigure}\qquad
\vspace{\baselineskip}
\caption{}\label{f7}
\end{center}
\end{figure}

\begin{proof}
By Proposition \ref{p6}, $M$ has exactly 6 fixed points. We label the fixed points by $q_1, \cdots, q_6$. By Proposition \ref{p2}, there is a multigraph $\Gamma$ describing $M$ that has no multiple edges and no self-loops. To each fixed point $q_i$ we assign a vertex, also denoted by $q_i$. Without loss of generality, by relabeling vertices, we may assume that for each $i \in \{2,3,4\}$, there is an edge between $q_1$ and $q_i$, as in Figure \ref{f7}.

Suppose that there is an edge $e$ from $q_i$ to $q_j$ with label $w$. Let $\Gamma'$ be a multigraph obtained from $\Gamma$ by reversing the direction of the edge $e$ and replacing the label $w$ of the edge $e$ with its negative $-w$. Then $\Gamma'$ also describes $M$. Therefore, by reversing the direction of an edge and replacing the label of the edge with its negative, we may assume that if there is an edge from $q_i$ to $q_j$, then $i<j$ (index increasing.) If there is an edge $e_{i,j}$ from $q_i$ to $q_j$, let $w_{i,j}$ be the label of the edge $e_{i,j}$.

We have following two cases.
\begin{enumerate}[(A)]
\item There is an edge between $q_2$, $q_3$, and $q_4$.
\item There is no edge between $q_2$, $q_3$, and $q_4$.
\end{enumerate}

Assume that Case (A) holds. Without loss of generality, by relabeling $q_2$, $q_3$, and $q_4$, we may assume that there is an edge from $q_2$ to $q_3$. Then we show that there does not exist an edge between $q_2$ and $q_4$. To show this, suppose on the contrary that there is an edge between $q_2$ and $q_4$. Since $\Gamma$ has no self-loops and no multiple edges, the three edges at $q_5$ must be connected to $q_3$, $q_4$, and $q_6$, respectively. But then $q_6$ must have a self-loop, which leads to a contradiction.

Similarly, by symmetry of $\Gamma$ there is no edge between $q_3$ and $q_4$. Then there must be an edge between $p_4$ and $p_5$ and an edge between $p_4$ and $p_6$. Thus, the remaining edge of $q_2$ is connected to either $q_5$ or $q_6$. Without loss of generality, assume that there is an edge between $q_2$ and $q_5$. If we complete drawing the multigraph with the required properties, the multigraph $\Gamma$ we get is Figure \ref{f1a}.

Assume that Case (B) holds. Then for each $i \in \{2,3,4\}$ and for each $j \in \{5,6\}$, there must be an edge between $q_i$ and $q_j$. The multigraph $\Gamma$ we get is Figure \ref{f1b}. \end{proof}

We show that Figure \ref{f1b} does not occur as a multigraph describing an almost complex torus manifold.

\begin{pro}\label{c1}
Figure \ref{f1b} cannot describe an almost complex torus manifold.
\end{pro}

\begin{figure}
\begin{center}
\begin{subfigure}[b][5cm][s]{.4\textwidth}
\centering
\begin{tikzpicture}[state/.style ={circle, draw}]
\node[vertex] (a) at (0, 0) {$q_{2}$};
\node[vertex] (b) at (-1.2, 1.7) {$q_{1}$};
\node[vertex] (c) at (0, 3.4) {$q_{3}$};
\node[vertex] (d) at (2, 3.4) {$q_{6}$};
\path (b) [->] edge node [left] {$a=w_{1,2}$} (a);
\path (b) [->] edge node [left] {$b=w_{1,3}$} (c);
\path (a) [->] edge node [right] {$w_{2,6}$} (d);
\path (c) [->] edge node [above] {$w_{3,6}$} (d);
\end{tikzpicture}
\caption{$\tilde{F_1}$}\label{f5a}
\end{subfigure}\qquad
\begin{subfigure}[b][5cm][s]{.5\textwidth}
\centering
\vspace{15mm}
\begin{tikzpicture}[state/.style ={circle, draw}]
\node[vertex] (a) at (0, 0) {$q_{2}$};
\node[vertex] (b) at (-1.2, 1.7) {$q_{1}$};
\node[vertex] (e) at (3.2, 1.7) {$q_{4}$};
\node[vertex] (f) at (2, 0) {$q_{5}$};
\path (b) [->] edge node [left] {$a=w_{1,2}$} (a);
\path (e) [->] edge node [right] {$w_{4,5}$} (f);
\path (a) [->] edge node [below] {$w_{2,5}$} (f);
\path (b) [->] edge node [above] {$c=w_{1,4}$} (e);
\end{tikzpicture}
\caption{$\tilde{F_2}$}\label{f5b}
\end{subfigure}\qquad
\begin{subfigure}[b][3.5cm][s]{.5\textwidth}
\centering
\begin{tikzpicture}[state/.style ={circle, draw}]
\node[vertex] (b) at (-1.2, 1.7) {$q_{1}$};
\node[vertex] (c) at (0, 3.4) {$q_{3}$};
\node[vertex] (d) at (2, 3.4) {$q_{6}$};
\node[vertex] (e) at (3.2, 1.7) {$q_{4}$};
\path (b) [->] edge node [left] {$b=w_{1,3}$} (c);
\path (e) [->] edge node [right]{$w_{4,6}$} (d);
\path (b) [->] edge node [below]{$c=w_{1,4}$} (e);
\path (c) [->] edge node [above] {$w_{3,6}$} (d);
\end{tikzpicture}
\vspace{1mm}
\caption{$\tilde{F_3}$}\label{f5c}
\end{subfigure}\qquad
\caption{}\label{f5}
\end{center}
\end{figure}

\begin{proof}
Assume on the contrary that Figure \ref{f1b} describes an almost complex torus manifold $M$; then $\dim M=6$ and there are 6 fixed points.

Let $w_{1,2}=a$, $w_{1,3}=b$, and $w_{1,4}=c$. Let $\tilde{S}_1 := \ker a \cap \ker b$ and consider the isotropy submanifold $\tilde{F}_1$ of $M^{\tilde{S}_1}$ that contains $q_1$. By Lemma \ref{p3}, $\tilde{F}_1$ is 4-dimensional almost complex torus submanifold of $M$.

By Proposition \ref{p7}, some 2-valent sub-multigraph $\Gamma_{\tilde{F}_1}$ of Figure \ref{f1b} describes the $T^3$-action on $\tilde{F}_1$. Because there is an edge from $q_1$ to $q_2$ (from $q_1$ to $q_3$) with label $a$ ($b$) that satisfies $g^a=1$ ($g^b=1$) for all $g \in \tilde{S}_1$, Proposition \ref{p7} (2) implies that $q_2 \in \tilde{F}_1$ ($q_3 \in \tilde{F}_1$.) Moreover, since $\Gamma_{\tilde{F}_1}$ is 2-valent, two edges of $q_2$ ($q_3$) must be in $\Gamma_{\tilde{F}_1}$. On the other hand, by Lemma \ref{p3}, $\Gamma_{\tilde{F}_1}$ contains at most 4 vertices. Thus, $\Gamma_{\tilde{F}_1}$ cannot contain $q_4$, and must contain exactly one of $q_5$ and $q_6$.

Without loss of generality, by vertical symmetry of Figure \ref{f1b}, we may assume that $\tilde{F}_1$ contains $q_6$. Then Figure \ref{f5a} describes the $T^3$-action on $\tilde{F}_1$. Applying Lemma \ref{l5} to the $T^3$-action on $\tilde{F}_1$ with $w=a$ ($w=b$), since $q_1$ and $q_2$ ($q_1$ and $q_3$) are in the same component of $M^{\ker a}$ ($M^{\ker b}$), the $T^3$-weights $\{a,b\}$ in $T_{q_1}\tilde{F}_1$ and $\{-a,w_{2,6}\}$ in $T_{q_2}\tilde{F}_1$ ($\{a,b\}$ in $T_{q_1}\tilde{F}_1$ and $\{-b,w_{3,6}\}$ in $T_{q_3}\tilde{F}_1$) are equal modulo $a$ (modulo $b$.) Thus $w_{2,6}=b+k_1 a$ for some integer $k_1$ ($w_{3,6}=a+k_2 b$ for some integer $k_2$.)

Second, let $\tilde{F}_2$ be the 4-dimensional component of $M^{\tilde{S}_2}$ that contains $q_1$, where $\tilde{S}_2 = \ker a \cap \ker c$. By Proposition \ref{p7} and Lemma \ref{p3}, some 2-valent sub-multigraph $\Gamma_{\tilde{F_2}}$ of Figure \ref{f1b} describes the $T^3$-action on $\tilde{F}_2$ and contains at most 4 vertices. Thus $\Gamma_{\tilde{F_2}}$ contains exactly one of $q_5$ and $q_6$. Suppose that $\Gamma_{\tilde{F_2}}$ contains $q_6$. Then the weights $\{-w_{4,6}, -b-k_1a\}$ in $T_{q_6}\tilde{F_2}$ are linear combinations of $a$ and $c$, which leads to a contradiction since $\{a,b,c\}$ form a basis of $\mathbb{Z}^3$. Thus $\tilde{F}_2$ contains $q_1$, $q_2$, $q_4$, and $q_5$; see Figure \ref{f5b}. Applying Lemma \ref{l5} to the $T^3$-action on $\tilde{F}_2$ with $w=c$, since $q_1$ and $q_4$ are in the same component of $M^{\ker c}$, the $T^3$-weights $\{a,c\}$ in $T_{q_1}\tilde{F}_2$ and $\{-c,w_{4,5}\}$ in $T_{q_4}\tilde{F}_2$ are equal modulo $c$. Thus $w_{4,5}=a+k_3 c$ for some integer $k_3$.

Third, let $\tilde{F}_3$ be the 4-dimensional component of $M^{\tilde{S}_3}$ that contains $q_1$, where $\tilde{S}_3 = \ker b \cap \ker c$. By Proposition \ref{p7} and Lemma \ref{p3}, some 2-valent sub-multigraph $\Gamma_{\tilde{F_3}}$ of Figure \ref{f1b} describes the $T^3$-action on $\tilde{F}_3$ and contains at most 4 vertices. Thus $\Gamma_{\tilde{F_3}}$ contains exactly one of $q_5$ and $q_6$. Suppose that $\tilde{F}_3$ contains $q_5$. Then the weights $\{-w_{3,5}, -a-k_3c\}$ in $T_{q_5}\tilde{F}_3$ are linear combinations of $b$ and $c$, which leads to a contradiction. Therefore, $\tilde{F}_3$ contains $q_6$; see Figure \ref{f5c}. Then the weights $\{-b,a+k_2b\}$ in $T_{q_3}\tilde{F}_3$ are linear combinations of $b$ and $c$, which leads to a contradiction. \end{proof}

Next, for a 6-dimensional almost complex torus manifold with Euler number 6, we study 4-dimensional almost complex torus submanifolds and their sub-multigraphs.

\begin{lemma}\label{c2}
Let $M$ be an almost complex torus manifold described by Figure \ref{f1a}. Then there exist 4-dimensional almost complex torus submanifolds $F_1, \cdots, F_5$ described by Figure \ref{f6a}, $\cdots$, \ref{f6e}, respectively.
\end{lemma}

\begin{figure}
\begin{center}
\begin{subfigure}[b][5.5cm][s]{.2\textwidth}
\centering
\begin{tikzpicture}[state/.style ={circle, draw}]
\node[vertex] (a) at (0, 0) {$q_{2}$};
\node[vertex] (b) at (-1.2, 1.7) {$q_{1}$};
\node[vertex] (c) at (0, 3.4) {$q_{3}$};
\path (b) [->] edge node [left] {$w_{1,2}$} (a);
\path (b) [->] edge node [left] {$w_{1,3}$} (c);
\path (a) [->] edge node [right] {$w_{2,3}$} (c);
\end{tikzpicture}
\caption{$F_1$}\label{f6a}
\vspace{\baselineskip}
\end{subfigure}\qquad
\begin{subfigure}[b][5.5cm][s]{.2\textwidth}
\centering
\begin{tikzpicture}[state/.style ={circle, draw}]
\node[vertex] (d) at (2, 3.4) {$q_{6}$};
\node[vertex] (e) at (3.2, 1.7) {$q_{4}$};
\node[vertex] (f) at (2, 0) {$q_{5}$};
\path (e) [->] edge node [right] {$w_{4,5}$} (f);
\path (e) [->] edge node [right] {$w_{4,6}$} (d);
\path (f) [->] edge node [left]{$w_{5,6}$} (d);
\end{tikzpicture}
\caption{$F_2$}\label{f6b}
\vspace{\baselineskip}
\end{subfigure}\qquad
\begin{subfigure}[b][4cm][s]{.4\textwidth}
\centering
\begin{tikzpicture}[state/.style ={circle, draw}]
\node[vertex] (a) at (0, 0) {$q_{2}$};
\node[vertex] (b) at (-1.2, 1.7) {$q_{1}$};
\node[vertex] (e) at (3.2, 1.7) {$q_{4}$};
\node[vertex] (f) at (2, 0) {$q_{5}$};
\path (b) [->] edge node [left] {$w_{1,2}$} (a);
\path (e) [->] edge node [right] {$w_{4,5}$} (f);
\path (a) [->] edge node [below] {$w_{2,5}$} (f);
\path (b) [->] edge node [above] {$w_{1,4}$} (e);
\end{tikzpicture}
\caption{$F_3$}\label{f6c}
\vspace{\baselineskip}
\end{subfigure}\qquad
\begin{subfigure}[b][4cm][s]{.4\textwidth}
\centering
\vspace{-10mm}
\begin{tikzpicture}[state/.style ={circle, draw}]
\node[vertex] (b) at (-1.2, 1.7) {$q_{1}$};
\node[vertex] (c) at (0, 3.4) {$q_{3}$};
\node[vertex] (d) at (2, 3.4) {$q_{6}$};
\node[vertex] (e) at (3.2, 1.7) {$q_{4}$};
\path (b) [->] edge node [left] {$w_{1,3}$} (c);
\path (e) [->] edge node [right] {$w_{4,6}$} (d);
\path (b) [->] edge node [below] {$w_{1,4}$} (e);
\path (c) [->] edge node [above] {$w_{3,6}$} (d);
\end{tikzpicture}
\vspace{12mm}
\caption{$F_4$}\label{f6d}
\vspace{\baselineskip}
\end{subfigure}\qquad\begin{subfigure}[b][5cm][s]{.4\textwidth}
\centering
\begin{tikzpicture}[state/.style ={circle, draw}]
\node[vertex] (a) at (0, 0) {$q_{2}$};
\node[vertex] (c) at (0, 3.4) {$q_{3}$};
\node[vertex] (d) at (2, 3.4) {$q_{6}$};
\node[vertex] (f) at (2, 0) {$q_{5}$};
\path (a) [->] edge node [left] {$w_{2,3}$} (c);
\path (f) [->] edge node [right]{$w_{5,6}$} (d);
\path (a) [->] edge node [below] {$w_{2,5}$} (f);
\path (c) [->] edge node [above] {$w_{3,6}$} (d);
\end{tikzpicture}
\caption{$F_5$}\label{f6e}
\vspace{\baselineskip}
\end{subfigure}\qquad
\caption{}\label{f6}
\end{center}
\end{figure}

\begin{proof}
The idea of proof is similar to the one of Proposition \ref{c1}.

Let $F_1$ be the 4-dimensional component of $M^{\ker w_{1,2} \cap \ker w_{1,3}}$ that contains $q_1$. By Proposition \ref{p7}, some 2-valent sub-multigraph $\Gamma_{F_1}$ of Figure \ref{f1a} describes the $T^3$-action on $F_1$. Since there is an edge from $q_1$ to $q_2$ with label $w_{1,2}$, $q_2 \in F_1$ and similarly $q_3 \in F_1$. By Lemma \ref{p3}, $\Gamma_{F_1}$ contains at most 4 vertices. If $\Gamma_{F_1}$ contains one of $q_4$, $q_5$, and $q_6$, then $\Gamma_{F_1}$ must contain at least 5 vertices. Thus, $\Gamma_{F_1}$ is Figure \ref{f6a}. 

All the other cases are analogous. If we let $F_2$ ($F_3$, $F_4$, and $F_5$) be the 4-dimensional component of $M^{\ker w_{4,5} \cap \ker w_{4,6}}$ ($M^{\ker w_{1,2} \cap \ker w_{1,4}}$, $M^{\ker w_{1,3} \cap \ker w_{1,4}}$, and $M^{\ker w_{2,3} \cap \ker w_{2,5}}$) that contains $q_4$ ($q_1$, $q_1$, and $q_2$), then by Proposition \ref{p7} and Lemma \ref{p3}, Figure \ref{f6b} (Figure \ref{f6c}, Figure \ref{f6d}, and Figure \ref{f6e}) describes the $T^3$-action on $F_2$ ($F_3$, $F_4$, and $F_5$, respectively.) 

For each $1 \leq i \leq 5$, quotienting out by the subgroup of $T^3$ that acts trivially on $F_i$, it follows that $F_i$ is a 4-dimensional almost complex torus manifold.
\end{proof}

\section{Proof of Theorem \ref{t11}} \label{s7}

In this section, we prove Theorem \ref{t11}.
For this, we first determine weights at the fixed points, for a 4-dimensional almost complex torus manifold with at most 4 fixed points; by Theorem \ref{t211}, such a manifold has at least 3 fixed points.

\begin{lemma} \label{l6}
Let $M$ be a 4-dimensional almost complex torus manifold.
\begin{enumerate}[(1)]
\item If there are 3 fixed points, Figure \ref{f11a} describes $M$ for some $a$ and $b$ that span $\mathbb{Z}^2$.
\item If there are 4 fixed points, Figure \ref{f11b} describes $M$ for some $a$ and $b$ that span $\mathbb{Z}^2$ and for some integer $m$.
\end{enumerate}
\end{lemma}

\begin{figure}
\begin{center}
\begin{subfigure}[b][5cm][s]{.4\textwidth}
\centering
\begin{tikzpicture}[state/.style ={circle, draw}]
\node[vertex] (d) at (0, 3.4) {$p_{3}$};
\node[vertex] (e) at (-1.4, 1.7) {$p_{1}$};
\node[vertex] (f) at (0, 0) {$p_{2}$};
\path (e) [->] edge node [left] {$a$} (f);
\path (e) [->] edge node [left] {$a+b$} (d);
\path (f) [->] edge node [right]{$b$} (d);
\end{tikzpicture}
\caption{3 fixed points}\label{f11a}
\vspace{\baselineskip}
\end{subfigure}\qquad
\begin{subfigure}[b][5cm][s]{.3\textwidth}
\centering
\vspace{2mm}
\begin{tikzpicture}[state/.style ={circle, draw}]
\node[vertex] (a) at (0, 0) {$p_{1}$};
\node[vertex] (c) at (0, 3) {$p_{3}$};
\node[vertex] (d) at (3, 3) {$p_{4}$};
\node[vertex] (f) at (3, 0) {$p_{2}$};
\path (a) [->] edge node [left] {$b$} (c);
\path (f) [->] edge node [right]{$b+ma$} (d);
\path (a) [->] edge node [below] {$a$} (f);
\path (c) [->] edge node [above] {$a$} (d);
\end{tikzpicture}
\caption{4 fixed points}\label{f11b}
\vspace{\baselineskip}
\end{subfigure}\qquad
\caption{}\label{f11}
\end{center}
\end{figure}

\begin{figure}
\begin{center}
\begin{subfigure}[b][5cm][s]{.4\textwidth}
\centering
\begin{tikzpicture}[state/.style ={circle, draw}]
\node[vertex] (d) at (0, 3.4) {$p_{3}$};
\node[vertex] (e) at (-1.4, 1.7) {$p_{1}$};
\node[vertex] (f) at (0, 0) {$p_{2}$};
\path (e) [->] edge node [left] {$a$} (f);
\path (e) [->] edge node [left] {$c$} (d);
\path (f) [->] edge node [right]{$b$} (d);
\end{tikzpicture}
\caption{}\label{f12a}
\vspace{\baselineskip}
\end{subfigure}\qquad
\begin{subfigure}[b][5cm][s]{.4\textwidth}
\centering
\begin{tikzpicture}[state/.style ={circle, draw}]
\node[vertex] (a) at (0, 0) {$p_{1}$};
\node[vertex] (c) at (0, 3) {$p_{3}$};
\node[vertex] (d) at (3, 3) {$p_{4}$};
\node[vertex] (f) at (3, 0) {$p_{2}$};
\path (a) [->] edge node [left] {$b$} (c);
\path (f) [->] edge node [right]{$c$} (d);
\path (a) [->] edge node [below] {$a$} (f);
\path (c) [->] edge node [above] {$d$} (d);
\end{tikzpicture}
\caption{}\label{f12b}
\vspace{\baselineskip}
\end{subfigure}\qquad
\caption{}\label{f12}
\end{center}
\end{figure}

\begin{proof}

Suppose that $M$ has 3 fixed points $p_1$, $p_2$, and $p_3$. By Proposition \ref{p2}, Figure \ref{f12a} describes $M$ for some non-zero $a, b, c \in \mathbb{Z}^2$. In particular, $p_1$, $p_2$, and $p_3$ have weights $\{a,c\}$, $\{-a, b\}$, and $\{-b-c\}$, respectively.

Since $\int_M$ is a map from $H_{T^2}^i (M;\mathbb{Z})$ to $H^{i-4}(BT^2; \mathbb{Z})$ for all $i$, the image of the equivariant cohomology class 1 under the map $\int_M$ is zero, that is,
\begin{center}
$\displaystyle\int_M 1=0$.
\end{center}
By Theorem \ref{ab}, 
\begin{center}$\displaystyle \int_M 1 = \sum_{p \in M^{T^2}} \frac{1}{e_{T^2} (N_p M)}= \sum_{p \in M^{T^2}} \frac{1}{e_{T^2} (T_p M)} = \frac{1}{ac}+\frac{1}{-ab} + \frac{1}{(-b)(-c)}$.\end{center}
Therefore, $c=a+b$.

Next, suppose that $M$ has 4 fixed points $p_1$, $p_2$, $p_3$, and $p_4$. 
By Proposition \ref{p2}, Figure \ref{f12b} describes $M$ for some non-zero $a, b, c, d \in \mathbb{Z}^2$.
In particular, $p_1$, $p_2$, $p_3$, and $p_4$ have weights $\{a,b\}$, $\{-a,c\}$, $\{-b,d\}$, and $\{-c,-d\}$, respectively.

Applying Theorem \ref{ab} to the equivariant cohomology class 1,
\begin{center}$\displaystyle 0=\int_M 1 = \sum_{p \in M^{T^2}} \frac{1}{e_{T^2} (N_p M)}= \sum_{p \in M^{T^2}} \frac{1}{e_{T^2} (T_p M)} = \frac{1}{ab}+\frac{1}{-ac} + \frac{1}{-bd}+ \frac{1}{(-c)(-d)} = \left(\frac{1}{a}-\frac{1}{d}\right)\left(\frac{1}{b}-\frac{1}{c}\right)$.\end{center}
Therefore, $a=d$ or $b=c$.

Suppose that $a=d$. Because there is an edge from $p_1$ to $p_2$ with label $a$, $p_1$ and $p_2$ are in the same component $F$ of $M^{\ker a}$. Applying Lemma \ref{l5} to the $T^2$-action on $F$, taking $w=a$, the weights $\{a,b\}$ at $p_1$ and the weights $\{-a,c\}$ at $p_2$ are equal modulo $a$. Thus, $c=b+ma$ for some integer $m$.

Similarly, if $b=c$, it follows from Lemma \ref{l5} that $a=b+mb$ for some integer $m$; in this case, by relabeling fixed points and weights, this lemma follows. \end{proof}

From Lemma \ref{l6}, we obtain the following lemma.

\begin{lemma}\label{l4}
Suppose that Figure \ref{f1a} describes an almost complex torus manifold. Then the followings hold.
\begin{enumerate}
\item $w_{1,2}+w_{2,3}=w_{1,3}.$
\item $w_{4,5}+w_{5,6}=w_{4,6}.$
\item $w_{1,2}=w_{4,5}$ and $w_{1,4} \equiv w_{2,5} \mod w_{1,2}$, or $w_{1,4}=w_{2,5}$ and $w_{1,2} \equiv w_{4,5} \mod w_{1,4}$.
\item $w_{1,3}=w_{4,6}$ and $w_{1,4} \equiv w_{3,6} \mod w_{1,3}$, or $w_{1,4}=w_{3,6}$ and $w_{1,3} \equiv w_{4,6} \mod w_{1,4}$.
\item $w_{2,3}=w_{5,6}$ and $w_{2,5} \equiv w_{3,6} \mod w_{2,3}$, or $w_{2,5}=w_{3,6}$ and $w_{2,3} \equiv w_{5,6} \mod w_{2,5}$.
\end{enumerate}
\end{lemma}

\begin{proof}
By Lemma \ref{c2}, Figure \ref{f6a} describes the 4-dimensional almost complex torus manifold $F_1$. Thus, by Lemma \ref{l6}, $w_{1,2}+w_{2,3}=w_{1,3}$.

Similarly, since Figures \ref{f6b}, $\cdots$, \ref{f6e} describe the 4-dimensional almost complex torus manifolds $F_2, \cdots, F_5$, respectively, all the other cases follow from Lemma \ref{l6}. \end{proof}

With all of the above, we are ready to prove our main result, Theorem \ref{t11}.

\begin{proof}[Proof of Theorem \ref{t11}]
By Proposition \ref{p9} and Proposition \ref{c1}, Figure \ref{f1a} describes $M$. By Lemma \ref{c2}, there exist 4-dimensional almost complex torus submanifolds $F_1, \cdots, F_5$ described by Figures \ref{f6a}, $\cdots$, \ref{f6e}, respectively. For convenience, let $a:=w_{1,2}$, $b:=w_{2,3}$, and $c:=w_{1,4}$. By Lemma \ref{l4} (1), $w_{1,3} = a+b$. The fixed point $q_1$ has weights $\{a,a+b,c\}$ that span $\mathbb{Z}^3$, and hence $a,b,c$ span $\mathbb{Z}^3$.

By Lemma \ref{l4} (3) we have the following two cases.
\begin{enumerate}[(A)]
\item $w_{4,5} = a$ and $w_{2,5} = c+ka$ for some integer $k$.
\item $w_{2,5} = c$ and $w_{4,5} = a+kc$ for some integer $k$.
\end{enumerate}

Suppose that Case (A) holds. By Lemma \ref{l4} (4) we have the following two cases.
\begin{enumerate}[(i)]
\item $w_{4,6} = a+b$ and $w_{3,6} = c+l(a+b)$ for some integer $l$.
\item $w_{3,6} = c$ and $w_{4,6} = a+b+lc$ for some integer $l$.
\end{enumerate}

Suppose that Case (A-i) holds. By Lemma \ref{l4} (2) it follows that $w_{5,6}= b$. Since $w_{2,3}=w_{5,6}=b$, by Lemma \ref{l4} (5) it follows that $c+ka=w_{2,5} \equiv w_{3,6}=c+l(a+b) \mod w_{2,3}=b$. Since $a,b,c$ form a basis of $\mathbb Z^3$, we have $l=k$ and Figure \ref{f4a} describes $M$.

Suppose that Case (A-ii) holds. By Lemma \ref{l4} (2) it follows that $w_{5,6} = b+lc$. By Lemma \ref{l4} (5), (A-ii-1) $b=w_{2,3}=w_{5,6}=b+lc$ and $c+ka=w_{2,5} \equiv w_{3,6}=c \mod w_{2,3}=b$, or (A-ii-2) $c+ka=w_{2,5}=w_{3,6}=c$ and $b=w_{2,3} \equiv w_{5,6}=b+lc \mod w_{2,5}=c+ka$. Since $a,b,c$ form a basis of $\mathbb{Z}^3$, in the former case (A-ii-1) we have $l=0$ and $k=0$ and both Figure \ref{f4a} with $k=0$ and Figure \ref{f4b} with $k=l=0$ describe $M$, and in the latter case (A-ii-2) we have $k=0$ and Figure \ref{f4b} with $k=0$ describes $M$.

Suppose that Case (B) holds. By Lemma \ref{l4} (5) we have the following two cases.
\begin{enumerate}[(i)]
\item $w_{5,6} = b$ and $w_{3,6} = c+lb$ for some integer $l$.
\item $w_{3,6} = c$ and $w_{5,6} = b+lc$ for some integer $l$.
\end{enumerate}

Suppose that Case (B-i) holds. By Lemma \ref{l4} (2) it follows that $w_{4,6} = a+b+kc$. By Lemma \ref{l4} (4), (B-i-1) $a+b=w_{1,3}=w_{4,6}=a+b+kc$ and $c=w_{1,4} \equiv w_{3,6}=c+lb \mod w_{1,3}=a+b$, or (B-i-2) $c=w_{1,4}=w_{3,6}=c+lb$ and $a+b=w_{1,3} \equiv w_{4,6}=a+b+kc \mod w_{1,4}=c$. Because $a,b,c$ span $\mathbb{Z}^3$, in the former case (B-i-1) we have $k=l=0$ and both Figure \ref{f4a} with $k=0$ and Figure \ref{f4b} with $k=l=0$ describe $M$, and in the latter case (B-i-2) we have $l=0$ and Figure \ref{f4b} with $l=0$ describes $M$.

Suppose that Case (B-ii) holds. By Lemma \ref{l4} (2) it follows that $w_{4,6} = a+b+(k+l)c$. In this case, Figure \ref{f4b} describes $M$.
\end{proof}

\section{Invariants} \label{s8}

In this section, we prove Theorem \ref{t81}; we determine invariants of 6-dimensional almost complex torus manifolds with Euler number 6.

\begin{proof}[\textbf{Proof of Theorem \ref{t81}}]
By Proposition \ref{p6}, $M$ has 6 fixed points. Because $\dim M=6$, for a fixed point $p$, the third equivariant Chern class $c_3(M)$ at $p$ is the same as the equivariant Euler class of the normal space at $p$, which is the product of the weights at $p$. That is, $c_3(M)(p)=e_{T^3}(Np)=\prod_{i=1}^3 w_{p,i}$. Therefore, by Theorem \ref{ab}, the Chern number $\int_M c_3$ is
\begin{center}
$\displaystyle \int_M c_3=\sum_{p \in M^{T^3}} \frac{c_3(M)(p)}{e_{T^3}(Np)}=\sum_{p \in M^{T^3}} \frac{\prod_{i=1}^3 w_{p,i}}{\prod_{i=1}^3 w_{p,i}}=\sum_{p \in M^{T^3}}1=6$.
\end{center}

First, suppose that Figure \ref{f4a} describes $M$. 
The weights at the fixed points $p_1,\cdots,p_6$ are $\{-a,b,c+ka\}$, $\{a,a+b,c\}$, $\{-b, -a-b, c+k(a+b)\}$, $\{-c-k(a+b), -a-b, -b\}$, $\{-c, a+b, a\}$, $\{-c-ka, b, -a\}$, respectively.

For a non-negative integer $i$ and real numbers $x_1$, $x_2$, $\cdots$, $x_n$, let $\sigma_i(x_1,x_2,\cdots,x_n)$ denote the  $i$-th elementary symmetric polynomial in $x_j$. For instance, $\sigma_0(x_1,\cdots,x_n)=1$, $\sigma_1(x_1,\cdots,x_n)=\sum_{i=1}^n x_i$, and $\sigma_2(x_1,\cdots,x_n)=\sum_{i<j} x_i x_j$.

By Theorem \ref{ab},
\begin{center} $\displaystyle \int_M c_1 c_2 = \sum_{p\in M^{T^3}} \frac{\sigma_1 (w_{p,1}, w_{p,2}, w_{p,3}) \sigma_2 (w_{p,1}, w_{p,2}, w_{p,3})}{\prod_{i=1}^3 w_{p,i}}$
$\displaystyle=-\frac{\{(k-1)a+b+c\}\{-ka^2+(k-1)ab-ac+bc\}}{ab(c+ka)}$
$\displaystyle+\frac{(2a+b+c)(a^2+ab+2ac+bc)}{a(a+b)c}$
$\displaystyle-\frac{\{(k-1)a+(k-2)b+c\}\{ka^2+(3k-1)ab+ac+(2k-1)b^2+2bc\}}{b(a+b)\{c+k(a+b)\}}$
$\displaystyle+\frac{\{(k+1)a+(k+2)b+c\}\{ka^2+(3k+1)ab+ac+(2k+1)b^2+2bc\}}{\{c+k(a+b)\}(a+b)b}$
$\displaystyle-\frac{(2a+b-c)(a^2+ab-2ac-bc)}{c(a+b)a}$
$\displaystyle+\frac{\{(-k-1)a+b-c\}\{ka^2+(-k-1)ab+ac-bc\}}{(c+ka)ba}.$
\end{center}

Combining the terms that have the same denominator, this is equal to
\begin{center}
$\displaystyle\frac{2(-a^2+3ab-b^2)(ka+c)}{ab(c+ka)}
+\frac{2(5a^2+5ab+b^2)c}{a(a+b)c}$
$\displaystyle+\frac{2(a^2+5ab+5b^2)(ka+kb+c)}{(a+b)b\{c+k(a+b)\}}$
$\displaystyle=\frac{2}{ab(a+b)}\{(-a^2+3ab-b^2)\times(a+b) + (5a^2+5ab+b^2)\times b +(a^2+5ab+5b^2)\times a\}$

$\displaystyle=\frac{2}{ab(a+b)}\times 12ab(a+b)=24$.\end{center}
Therefore, $\int_M c_1 c_2 = 24$.

Since $\int_M T_0^3 = \int_M -T_3^3 = \int_M \frac{c_1 c_2}{24} = 1$ and $\int_M T_1^3 = \int_M -T_2^3 = \int_M \frac{c_1 c_2 - 12 c_3}{24}= -2$, the Hirzebruch $\chi_y$-genus of $M$ is
\begin{center}
$\displaystyle \chi_y (M) = \sum_{i=0}^3 \left( \int_M T_i^3 \right)y^i = 1-2y + 2y^2-y^3$.
\end{center}

Next we compute $\int_M c_1^3$. By Theorem \ref{ab},
\begin{center}$\displaystyle \int_M c_1^3=\sum_{p\in M^{T^3}} \frac{\sigma_1(w_{p,1}, w_{p,2}, w_{p,3})^3}{\prod_{i=1}^3 w_{p,i}}$
$\displaystyle=-\frac{\{(k-1)a+b+c\}^3}{ab(c+ka)}
+\frac{(2a+b+c)^3}{a(a+b)c}
+\frac{\{(k-1)a+(k-2)b+c\}^3}{b(a+b)\{c+k(a+b)\}}
-\frac{\{(-k-1)a+(-k-2)b-c\}^3}{\{c+k(a+b)\}(a+b)b}
-\frac{(2a+b-c)^3}{c(a+b)a}
+\frac{\{(-k-1)a+b-c\}^3}{(c+ka)ba}$.
\end{center}
Combining the terms that have the same denominator, this is equal to
\begin{center}
$\displaystyle 
\frac{-2\{(k^2+3)a^2-6ab+2kac+3b^2+c^2\}}{ab}$
$\displaystyle
+\frac{2(12a^2+12ab+3b^2+c^2)}{a(a+b)}$
$\displaystyle+\frac{2\{(k^2+3)a^2+(2k^2+12)ab+2kac+(k^2+12)b^2+2kbc+c^2\}}{b(a+b)}$

$\displaystyle= \frac{2}{ab(a+b)}\times [-\{(k^2+3)a^2-6ab+2kac+3b^2+c^2\}\times (a+b)$
$\displaystyle+(12a^2+12ab+3b^2+c^2) \times b
+\{(k^2+3)a^2+(2k^2+12)ab+2kac+(k^2+12)b^2+2kbc+c^2\}\times a]$
$\displaystyle =\frac{2}{ab(a+b)} \times (k^2+27)ab(a+b)=2(k^2+27)$.\end{center}

Second, suppose that Figure \ref{f4b} describes $M$.
The weights at the fixed points $p_1,\cdots,p_6$ are $\{-a, b, c\}$, $\{a, a+b, c\}$, $\{-b, -a-b, c\}$, $\{-c, -a-b-(k+l)c, -b-lc\}$, $\{-c, a+b+(k+l)c, a+kc\}$, $\{-c, b+lc, -a-kc\}$, respectively. 
We shall replace $a+kc$, $b+lc$, and $-c$ by $A$, $B$, and $C$, respectively. Then the weights at the latter three fixed points are $\{C, -A-B, -B\}$, $\{C, A+B, A\}$, $\{C, B, -A\}$, which are of the same form as the former three fixed points.
By Theorem \ref{ab},
\begin{center}$\displaystyle \int_M c_1 c_2 = \sum_{p\in M^{T^3}} \frac{\sigma_1 (w_{p,1}, w_{p,2}, w_{p,3}) \sigma_2 (w_{p,1}, w_{p,2}, w_{p,3})}{\prod_{i=1}^3 w_{p,i}}$
$\displaystyle=
\frac{(-a+b+c)(-ab-ac+bc)}{-abc}
+\frac{(2a+b+c)(a^2+ab+ac+bc)}{a(a+b)c}
+\frac{(-a-2b+c)(ab-ac+b^2-2bc)}{b(a+b)c}$
$\displaystyle+\frac{(-A-2B+C)(AB-AC+B^2-2BC)}{B(A+B)C}$
$\displaystyle+\frac{(2A+B+C)(A^2+AB+AC+BC)}{A(A+B)C}$
$\displaystyle+\frac{(-A+B+C)(-AB-AC+BC)}{-ABC}
$.\end{center}

Computing the first three terms, we get
\begin{center}$\displaystyle \frac{1}{ab(a+b)c} \times
\{-(-a+b+c)(-ab-ac+bc)\times(a+b)
+(2a+b+c)(a^2+ab+ac+bc)\times b
+(-a-2b+c)(ab-ac+b^2-2bc)\times a\}$
$\displaystyle=\frac{1}{ab(a+b)c}\times 12abc(a+b)=12.$\end{center}

Hence the sum of the latter three terms is also 12. Therefore, $\int_M c_1 c_2 = 24$. As in the first case, $\int_M T_0^3 = \int_M -T_3^3 = \int_M \frac{c_1 c_2}{24} = 1$ and $\int_M T_1^3 = \int_M -T_2^3 = \int_M \frac{c_1 c_2 - 12 c_3}{24}= -2$ and so the Hirzebruch $\chi_y$-genus of $M$ is
\begin{center}
$\displaystyle \chi_y (M) = \sum_{i=0}^3 \left( \int_M T_i^3 \right) y^i = 1-2y + 2y^2-y^3$.
\end{center}

We compute $\int_M c_1^3$. By Theorem \ref{ab},

\begin{center}$\displaystyle \int_M c_1^3=\sum_{p\in M^{T^3}} \frac{\sigma_1(w_{p,1}, w_{p,2}, w_{p,3})^3}{\prod_{i=1}^3 w_{p,i}}$
$=\displaystyle\frac{(-a+b+c)^3}{-abc}
+\frac{(2a+b+c)^3}{a(a+b)c}
+\frac{(-a-2b+c)^3}{b(a+b)c}
+\frac{(-A-2B+C)^3}{B(A+B)C}
+\frac{(2A+B+C)^3}{A(A+B)C}
+\frac{(-A+B+C)^3}{-ABC}$
.\end{center}

Computing the first three terms, we get
\begin{center}
$\displaystyle\frac{1}{abc(a+b)}\times \{ -(-a+b+c)^3\times(a+b)+(2a+b+c)^3\times b+(-a-2b+c)^3\times a\}$

$\displaystyle=\frac{1}{abc(a+b)} \times 27abc(a+b) =27$.
\end{center}
Thus, the sum of the latter three terms is also 27. Therefore, $\int_M c_1^3=54$. 
\end{proof}

\end{document}